\documentclass{amsart}

\newtheorem{theorem}{Theorem}[section]
\newtheorem{lemma}[theorem]{Lemma}
\newtheorem{assumption}[theorem]{Assumption}

\theoremstyle{definition}
\newtheorem{definition}[theorem]{Definition}

\theoremstyle{remark}

\numberwithin{equation}{section}

\usepackage{graphicx}
\usepackage{tabularx}
\usepackage{algorithm}
\usepackage{color}
\usepackage{enumitem}
\usepackage{bbm}
\usepackage{cancel}

\newcommand{\Pbk}{\mathbb{P}_{k-1}}
\newcommand{\Pbkd}{\mathbb{P}_{k-1/2}}
\newcommand{\R}{\mathbb{R}}
\newcommand{\N}{\mathbb{N}}

\newcommand{\Ek}{\mathbb{E}_{k-1}}

\renewcommand\refname{References}

\newcommand{\Predk}{\operatorname{Pred}_k}
\newcommand{\Aredk}{\operatorname{Ared}_k}

\newcommand{\expect}{\mathbb{E}}
\newcommand{\nomepunto}{{\sc irerm}}
\newcommand{\nome}{{\sc irerm }}

\def \indic{{\mathbbm 1}}

\begin{document}

\title[Inexact Restoration via random models]{Inexact Restoration via random models for unconstrained noisy optimization 
}

\author{Benedetta Morini}
\address{Dipartimento di Ingegneria Industriale,
    Universit\`{a} degli Studi di Firenze, Via G.B. Morgagni 40, 50134, Firenze, Italy.}
\email{benedetta.morini@unifi.it}

\author{Simone Rebegoldi}
\address{Dipartimento di Scienze Fisiche, Informatiche e Matematiche, Universit\`a di Modena e Reggio Emilia, Via Campi 213/b, 41125, Modena, Italy.}
\email{simone.rebegoldi@unimore.it}
\thanks{ The authors are members of the INdAM Research Group GNCS. Benedetta Morini and Simone Rebegoldi are partially supported by INDAM-GNCS through Progetti di Ricerca 2025 (CUP E53C24001950001). Benedetta Morini and Simone Rebegoldi are partially supported by PNRR - Missione 4 Istruzione e Ricerca - Componente C2 Investimento 1.1, Fondo per il Programma Nazionale di Ricerca e Progetti di Rilevante Interesse Nazionale (PRIN) funded by the European Commission under the NextGeneration EU programme, project ``Advanced optimization METhods for automated central veIn Sign detection in multiple sclerosis from magneTic resonAnce imaging (AMETISTA)'',  code: P2022J9SNP, MUR D.D. financing decree n. 1379 of 1st September 2023 (CUP E53D23017980001). Benedetta Morini is partially supported by PNRR - Missione 4 Istruzione e Ricerca - Componente C2 Investimento 1.1, Fondo per il Programma Nazionale di Ricerca e Progetti di Rilevante Interesse Nazionale (PRIN) funded by the European Commission under the NextGeneration EU programme, project ``Numerical Optimization with Adaptive Accuracy and Applications to Machine Learning'',  code: 2022N3ZNAX, MUR D.D. financing decree n. 973 of 30th June 2023 (CUP B53D23012670006), and by Partenariato esteso FAIR ``Future Artificial Intelligence Research'' SPOKE 1 Human-Centered AI. Obiettivo 4, Project ``Mathematical and Physical approaches to innovative Machine Learning technologies (MaPLe)'',  code: EP\_FAIR\_002, CUP B93C23001750006. 
}
\subjclass[2020]{65K05, 90C30, 90C15}

\keywords{Trust-region algorithms, random models, Inexact Restoration.}

\begin{abstract}
We study the Inexact Restoration framework with random models for minimizing functions whose evaluation
is subject to errors. We propose a constrained formulation that includes well-known stochastic 
problems and an algorithm applicable when the evaluation of 
both the function and its gradient is random and a specified accuracy of such evaluations is guaranteed with sufficiently high probability. 
The proposed algorithm combines the Inexact Restoration framework with a trust-region methodology based on random first-order models.
 We analyse the properties of the algorithm and provide the expected number of iterations performed to reach an approximate first-order optimality point. Numerical experiments show that the proposed algorithm compares well with a state-of-the-art competitor. 
\end{abstract}

\maketitle


.

\section{Introduction}\label{sec1}
We consider unconstrained optimization problems of the form
\begin{equation}\label{eq:min}
\min_{x\in\mathbb{R}^n}f(x),
\end{equation}
where $f:\mathbb{R}^n\rightarrow \mathbb{R}$ is a smooth  function whose evaluation
is inexact. In particular, we are interested in problems where the evaluation of both the objective function
and its gradient is random and sufficient accuracy in the estimates can be guaranteed with sufficiently high probability \cite{CS, Spall}. 
This class of problems includes the minimization of either functions which are intrinsically noisy or expensive functions which are approximately computed, see, e.g., \cite{Bollapragada2018, Bottou, ChenMeniSche18}.

Our proposal for the numerical solution of this problem is based on the Inexact Restoration framework
for constrained optimization problems proposed in \cite{MP}. Inexact Restoration is
a two-phase method consisting of a restoration phase and an optimization phase; 
the restoration phase improves feasibility and does not call for function evaluations, whereas the optimization phase
improves the objective function value with respect to the point obtained in the restoration phase, see e.g., \cite{AEP, BM, BHM, FF}. Such a framework has been successfully employed in the minimization of functions whose computation 
is intrinsically inexact; specifically, the exact evaluation of the function is interpreted as feasibility of a
constrained problem and the restoration phase increases the accuracy in the function evaluation \cite{KM, BKM18, BKM20_1}. Analogously, Inexact Restoration has been fruitfully applied in finite-sum minimization with inexact functions 
and derivatives \cite{BKM20, Bellavia2023, BMR}.

In order to apply the Inexact Restoration approach, we reformulate (\ref{eq:min}) as  
\begin{equation}\label{eq:min_vinc}
\min_{x\in\mathbb{R}^n}\bar{f}(x,y, \omega_y)\quad \mbox{ subject to } h(y)=0,
\end{equation}
where $\bar{f}(x,y, \omega_y)$ is an estimate of $f(x)$   with noise level $y\in Y$, $\omega_y\in \Omega$ denotes a random variable dependent of $y$ and independent of $x$, and $h:Y\rightarrow \R$ is a nonnegative function such that
the lower $h(y)$, the higher the accuracy of the function and gradient estimates is in probability.
In the ideal case $h(y)=0$, the function estimate $\bar{f}(x,y, \omega_y)$ coincides with the exact objective value $f(x)$. 
We show that our problem setting is viable for well-known optimization problems and then 
design and analyze a new  trust-region algorithm that employs random first-order models for optimization. We borrow ideas 
from the Inexact Restoration approach and adapt such framework, in order to define an algorithm 
that is well-defined and  theoretically well founded even if the random models 
are inaccurate at some iterations. We denote our algorithm  as \nome (Inexact REstoration with Random Models).

Our work is different from existing works on the Inexact Restoration approach for the minimization of noisy functions  \cite{BLM, KM, BKM18, BKM20_1}, where either the noise is controllable in a deterministic way or the noise is stochastic but the probability 
of obtaining sufficiently accurate estimates of functions and gradients is increasing;  by contrast, we assume that the noise is stochastic and the probability of obtaining accurate estimates is prefixed.  A stochastic first-order trust-region algorithm with Inexact Restoration has been proposed in \cite{Bellavia2023}, but only for finite-sum minimization problems; by contrast, \nome can be applied to any unconstrained differentiable problem of the form \eqref{eq:min}.
Our algorithm falls in the class of trust-region procedures with random models, see e.g., the contributions
\cite{BSV, BGMT2, bcms,ChenMeniSche18,Larson2016}.  With respect to the existing literature, \nome is inspired by {\sc storm}, the trust-region algorithm with random models proposed in \cite{ChenMeniSche18} and further analysed in \cite{bcms}. Differently from {\sc storm}, the probabilistic accuracy requirements 
on function and gradient  estimates are explicitly imposed  within \nome using problem  (\ref{eq:min_vinc})
and a suitable function $h$.  Furthermore, the acceptance test in \nome is a stochastic counterpart of the rule typically adopted in the Inexact Restoration framework and involves both  $h$ and the noisy function $\bar{f}$, whereas  the acceptance test in 
{\sc storm} is a stochastic counterpart of the  classical trust-region acceptance test and involves  only the noisy values of $f$.
 
We design the new algorithm {\sc irerm} and  show that it shares theoretical properties analogous to those of  {\sc storm}.  Particularly, we provide a bound on the expected number of iterations needed to achieve a certain accuracy in the norm of the true gradient. Furthermore, we compare numerically {\sc irerm} with {\sc storm} on a collection of least-squares problems and observe that 
{\sc irerm}  achieves comparable or lower noiseless values of $f$ than {\sc storm}.

The rest of the paper is organized as follows. In Section \ref{S_PF} we 
discuss our constrained problem formulation (\ref{eq:min_vinc}) and  present fields of applications for our approach.
In Section \ref{S_algo} we introduce our procedure and in Section \ref{S_dim} we study its theoretical properties 
and provide the expected number of iterations performed to reach an approximate first-order optimality point for (\ref{eq:min}).
Finally, in Section \ref{S_exps} we provide the numerical validation of our new algorithm.

\section{Problem formulation}\label{S_PF}
In this section, we discuss the assumptions made on problem (\ref{eq:min}) and its constrained formulation (\ref{eq:min_vinc}) and show that they are viable.

For any $x\in\R^n$ and level of noise $y\in Y$, let $\bar{f}(x,y,\omega_y)$ and $\nabla_x \bar{f}(x,y,\omega_y)$ be random evaluations of $f(x)$ and $\nabla  f(x)$, respectively, where $\omega_y\in\Omega$ denotes a random variable from a probability space with probability measure $\mathbb{P}$, dependent of $y$ and independent of $x$, that models the noise on the function and its gradient. Regarding the constraint in \eqref{eq:min_vinc}, the function $h:Y\rightarrow \R$ is assumed to be non-negative. 
If $h(y)=0$, then   $\bar f(x,y, \omega_y)$ and $\nabla_x \bar{f}(x,y, \omega_y)$ coincide with $f(x)$ and $\nabla f(x)$, respectively.
If $h(y)$ is non-zero, its value is related to the level of accuracy that  $\bar f(x,y, \omega_y)$ and $\nabla_x \bar{f}(x,y, \omega_y)$ can  achieve in probability  
with respect to the corresponding exact evaluations. We summarize these assumptions in the following.

\vskip 5pt
\begin{assumption}\label{ass_fh}
\vskip 1pt \noindent
\begin{itemize}
\item[(i)] There exist a function $h:Y\rightarrow \R$ and $h_{up}\geq 0$ such that $0\leq h(y)\leq h_{up}$ for all $y\in Y$. 
If $h(y)=0$, then $\bar f(x,y,\omega_y)=f(x)$  for all $x\in \R^n$.
\item[(ii)]  Given $x\in\R^n$, there exists a function $\bar{\mu}:(0,1]\rightarrow \R$,  $\bar{\mu}(\alpha)\in [0, \bar{\mu}_{up}]$ for all $\alpha\in(0,1]$  and for some $\bar{\mu}_{up}>0$,  such that, if $y\in Y$ satisfies $h(y)\leq \bar{\mu}(\alpha)\rho$, with $\alpha\in (0,1]$, $\rho>0$, then it holds
\begin{equation}\label{eq:prob_requirements}
\mathbb{P}(|f(x)- \bar f(x,y,\omega_y)|\le \rho)\ge \alpha, \quad \mathbb{P}(\|\nabla f(x)- \nabla_x \bar{f}(x,y,\omega_y)\|\le 
\rho)\ge \alpha.
\end{equation}
\item[(iii)] Functions $f(x)$ and $\bar{f}(x,y,\omega_y)$ are bounded from below, i.e., there exists $f_{low}$ such that
\begin{align*}
f(x)&\ge f_{low}, \quad \forall \ x\in\R^n\\
\bar f(x,y,\omega_y)&\ge f_{low}, \quad \forall \ x\in\R^n, \ \forall \ y\in Y, \ \forall  \ \omega_y\in \Omega.
\end{align*}
\end{itemize}

\end{assumption}
\vskip 5pt
Our problem setting includes, but it is not limited to,  the following three settings of stochastic noise in unconstrained optimization.

First, in finite-sum minimization problems, where the objective function in \eqref{eq:min} has the form
\begin{equation}\label{eq:finite-sum}
f(x)=\frac 1 N \sum_{i=1}^N \phi_i(x),
\end{equation}
and $N$ is a large positive integer, 
approximations of $f$ and $\nabla f$ can be computed by sampling. This amounts to
fixing the sample size $y$ as an integer in $Y=\{1,\ldots,N\}$, choosing the sample set of indices $I_y \subset
\{1, \ldots, N\}$ of cardinality $y$ at random, and computing
\begin{equation}\label{eq:approximations1}
\bar f(x,y,\omega_y)=\frac{1}{y}\sum_{i\in I_{y}} \phi_i(x), \quad 
\nabla_x  \bar{f}(x,y,\omega_y)=
\frac{1}{ y}\sum_{i\in I_{y}} \nabla \phi_i(x).
\end{equation}
By an abuse of notation, here  $\omega_y$  is a realization of the noise, represented by the random selection of a subset $I_{y}$ of cardinality $y$ containing the sample indices. Such evaluations provide the exact values $f(x)$ and $\nabla f(x)$ if $y=N$, while they are sufficiently 
accurate in probability if $y\in \{1, \ldots,N-1\}$  is large enough \cite[\S 6]{Tropp}.
Thus, problem \eqref{eq:min} with (\ref{eq:finite-sum}) as objective function can be reformulated 
as  \eqref{eq:min_vinc} with  
\begin{equation*}
h(y)=\begin{cases}
\sqrt{\frac{N}{y}}, \quad &\text{if }1\le y\le N-1\\
0, \quad &\text{if } y=N.
\end{cases}
\end{equation*}
Regarding the fulfillment of Assumption \ref{ass_fh}(ii), 
we assume that the noise in function  evaluations is unbiased and we 
let  $\operatorname{Var}_{\omega_y}[\bar{f}(x,y,\omega_y)]\leq V<\infty$.
Then, given $\rho>0$, we choose $y$ such that 
\begin{equation}\label{eq:h2}
h(y)\leq \sqrt{\frac{N(1-\alpha)}{V}}\rho.
\end{equation}
If $y=N$, such condition is satisfied for any $\rho>0$ being $h(N)=0$, 
and the probabilistic inequalities in \eqref{eq:prob_requirements} trivially hold.
If $y<N$, condition (\ref{eq:h2}) is equivalent to
\begin{equation}\label{eq:y}
y\geq \frac{V}{(1-\alpha)\rho^2}.
\end{equation}
and $\bar{f}(x,y,\omega_y)$  satisfies \eqref{eq:prob_requirements}, see \cite[\S 6]{Tropp}. The same reasoning applies to the evaluation of  $\nabla_x\bar{f}(x,y,\omega_y)$ 
letting $V$ be the upper bound for the variance $\operatorname{Var}_{\omega_y}[ \nabla_{x}\bar{f}(x,y,\omega_y)]$. Thus, Assumption \ref{ass_fh}(ii) holds with $\bar{\mu}(\alpha)=\sqrt{N(1-\alpha)/V}$ and  $\bar \mu_{up}=\sqrt{N/V}$.

Second, consider the case where the objective function in \eqref{eq:min} is the expectation of a 
stochastic function that depends on a random  variable $\xi$, i.e.,
\begin{equation}\label{eq:stochastic}
f(x)=\expect_{\xi}[f(x,\xi)], \, \quad \forall \ x\in\R^n.
\end{equation}
Sample averaging approximations for 
reducing the variance on the evaluations can be used, and we set
\begin{eqnarray}
\bar{f}(x,y,\omega_y) &=& \begin{cases}
f(x), \quad &\text{if }y=0\\
\displaystyle \frac{1}{p(y)}\sum_{i=1}^{p(y)}f(x,\xi_i), \quad &\text{if }y\in(0,1],
\end{cases} \label{eq:approximations3}
\\
\nabla_x \bar{f}(x,y,\omega_y)&=&\begin{cases}
\nabla f(x), \quad &\text{if }y=0\\
\displaystyle \frac{1}{p(y)}\sum_{i=1}^{p(y)}\nabla _xf(x,\xi_i), \quad &\text{if }y\in(0,1],
\end{cases}\label{eq:approximations2}
\end{eqnarray}
where the symbol $p(y)=\lceil 1/y\rceil\in \mathbb{N}$ represents the number of samples for a given $y\in (0,1] $ and $\omega_y$ represents the i.i.d.~noise realizations of $\xi$, $i=1, \ldots, p(y)$. 
In this case, sufficient accuracy in probability is attained in the approximations 
$\bar{f}(x,y,\omega_y), \nabla_x \bar{f}(x,y,\omega_y)$ by choosing $y$  small enough \cite[\S 6]{Tropp}  
and problem (\ref{eq:min}) can be cast into (\ref{eq:min_vinc}), for instance by setting $Y=[0,1]$ and $h(y)=\sqrt{y}$,  $y\in Y$. 

As for Assumption \ref{ass_fh}(ii), we let  $\operatorname{Var}_{\omega_y}[\bar{f}(x,y,\omega_y)]\leq V<\infty$ 
for all $x\in\R^n$ and suppose that the noise is unbiased. Given $\rho>0$, choose $y\in Y$ such that
\begin{equation}\label{eq:hy1}
h(y)\leq \sqrt{\frac{1-\alpha}{V}}{\rho}.
\end{equation}
If $y=0$, then the above condition is a satisfied, and the first condition in \eqref{eq:prob_requirements} trivially holds, as $\bar{f}(x,y, \omega_y)=f(x)$ according to \eqref{eq:approximations3}. Otherwise, condition \eqref{eq:hy1} can be rewritten as
$ \frac{1}{y}\geq \frac{V}{(1-\alpha)\rho^2}$, which implies
\begin{equation}\label{eq:py}
p(y)\geq \frac{V}{(1-\alpha)\rho^2}.
\end{equation}

The estimator $\bar{f}(x,y,\omega_y)$ in \eqref{eq:approximations2} satisfies  the first probabilistic condition in 
\eqref{eq:prob_requirements} if  the sample size $p(y)$ satisfies \eqref{eq:py}, see 
\cite[\S 6]{Tropp}.
An analogous result holds for the second probabilistic condition in \eqref{eq:prob_requirements}, under the assumption that the noise in the gradient computation is unbiased, and the variance of the gradient estimate is upper bounded by $V$. Thus, Assumption \ref{ass_fh}(ii) holds with $ \bar{\mu}(\alpha) =\sqrt{(1-\alpha)/V} $, $\bar \mu_{up}=\sqrt{1/V}$.

Finally, problem (\ref{eq:min_vinc})  includes the least-squares problem
\begin{equation*}
f(x)=\frac{1}{2} \|F(x)\|_2^2,
\end{equation*}
associated to a nonlinear system $F(x)=0$  where both $F$ and its Jacobian are
computed via Monte Carlo simulations, see e.g.,  \cite{WCK}. Evaluations are made by using a number of trials as in (\ref{eq:approximations3}) and (\ref{eq:approximations2}) and enlarging the trial set increases the accuracy as in the case discussed above.

\section{An inexact restoration trust region method with random models}\label{S_algo}
In the following, we introduce our algorithm  named  \nome (Inexact REstoration with Random Models).
\nome is a stochastic variant of the classical trust-region algorithm, where suitable approximations of the objective function and its gradient are employed, and whose acceptance criterion relies on the decrease of a convex combination of the function estimate with the infeasibility measure $h$. 
Now we refer to the algorithm \nome stated below and provide a detailed  description of it.

\begin{algorithm}[h!]
{\bf Algorithm \nome: Inexact REstoration with Random Models}
\par\noindent\rule{\textwidth}{0.5pt}
\begin{flushleft}
{\bf Input:} $x_0\in\mathbb{R}^n$, $y_0\in Y$,  
$\eta_1\in(0,1)$, $\eta_2>0$, $\theta_0\in(0,1)$,  $\underline{\theta}\in (0,\theta_0)$, $r\in(0,1)$,  $\mu>0$,  $\gamma>1$, $\delta_{\max}>0$, $ \delta_0\in(0,\delta_{\max}]$.
\begin{itemize}[leftmargin=13pt]
\item[\bf 0.] {\bf Initialization} \item[] Set $k=0$.
\item[\bf 1.] {\bf Accuracy variables}\\
\item[] Choose  $\widetilde{y}_{k+1}$, $y_{k+1}^t$ and $y_{k+1}^g$ in $Y$ satisfying
\begin{equation}\label{updateh}
h(\widetilde{y}_{k+1})\le r   
h(y_{k}), \, \, h( {y}^t_{k+1})\le \mu \min\{\delta_k^2, h(y_{k})\},
\,\,  h(y_{k+1}^g)\le \mu \delta_k.
\end{equation}
\item[\bf 2.] {\bf Trust-region model}\\
\item[] Compute  
$\bar f_k^{\dagger}=\bar{f}(x_k, y_{k+1}^t,\omega_{y_{k+1}^t}^\dagger)$, $g_k=\nabla_{x} \bar{f}(x_k,y_{k+1}^g,\omega_{y_{k+1}^g})$, set 
\begin{equation*}
p_k=-\delta_k\frac{g_k}{\|g_k\|},
\end{equation*}
and build
$$m_k(p)= \bar f_k^{\dagger} +g_k^Tp.$$
\item[\bf 3.] {\bf Penalty parameter}\\
\item[]  Compute  $\bar f_k^*=\bar{f}(x_k,y_{k+1}^t,\omega_{y_{k+1}^t}^*)$ and evaluate $\Predk(\theta_k)$ as in \eqref{eq:pred}.
\\
\item[] Compute the penalty parameter $ \theta_{k+1}^t$ as
\begin{equation}\label{tkp1}
\theta_{k+1}^t=
\left\{
\begin{array}{ll}
\theta_k,  \quad & \mbox{if} \  \Predk(\theta_k)\ge\theta_k \delta_k\|g_k\|\\
\displaystyle  \frac{h(y_k)-h(\widetilde{y}_{k+1})}{  \bar f_k^{\, \dagger}-\bar f_k^*  +h(y_k)-h(\widetilde{y}_{k+1})}, \quad  &\mbox{otherwise}.
\end{array}
\right.
\end{equation}
\item[\bf 4.] {\bf Acceptance test}\\
\item[]  Compute  $\bar{f}_k^p = \bar{f}(x_k+p_k,y_{k+1}^t,\omega_{y_{k+1}^t}^p)$, evaluate $\Aredk(x_k+p_k,\theta_{k+1}^t)$ in \eqref{eq:ared} and $\Predk(\theta_{k+1}^t)$ as in \eqref{eq:pred}.\\
\item[] If $\Aredk(x_k+p_k,\theta_{k+1}^t)\geq \eta_1\Predk(\theta_{k+1}^t)$, $\|g_k\|\geq \eta_2\delta_k$ , $\theta_{k+1}^t\ge \underline \theta$   \textbf{(success)}, set
\begin{equation}\label{eq:succ}
x_{k+1} =x_k+p_k, \quad  y_{k+1} =y_{k+1}^t,  \quad 
\delta_{k+1} =\min\{\gamma \delta_k,\delta_{\max}\}, \quad \theta_{k+1}=\theta_{k+1}^t.
\end{equation}
Else \textbf{(unsuccess)} set 
\begin{equation}\label{eq:unsucc}
x_{k+1} =x_k, \quad y_{k+1} =y_{k}, \quad \delta_{k+1} = \delta_k/\gamma, \quad \theta_{k+1}=\theta_k.
\end{equation}
\item[\bf 5.] {\bf Iteration counter}
\item[] Set $k=k+1$ and go to Step 1.
\end{itemize}
\end{flushleft}
\end{algorithm}

At each iteration $k\geq 0$, we have at our disposal the iterate $x_k\in\R^n$, the accuracy variable $y_k\in Y$, the trust-region radius $\delta_k>0$, and the penalty parameter $\theta_k\in(0,1)$. 

At Step 1,  the accuracy variables $\widetilde{y}_{k+1}\in Y$, $y_{k+1}^t\in Y$, and $y_{k+1}^g\in Y$ are selected according to condition (\ref{updateh}). The variable $\widetilde{y}_{k+1}$ is chosen so that the value $h(\widetilde{y}_{k+1})$ is smaller than a fraction of the current value $h(y_k)$; the variable $y_{k+1}^t$ is selected so that $h(y_{k+1}^t)$ is upper bounded by a quantity proportional to the minimum value between $h(y_k)$ and the squared trust-region radius $\delta_k^2$; the variable $y_{k+1}^g$   is such that $h(y_{k+1}^g)$ is upper bounded by a quantity proportional to $\delta_k$.  As shown in the next section, condition (\ref{updateh}) is crucial to provide sufficiently accurate function and derivative approximations   in probability. 
We observe that the variable $\widetilde{y}_{k+1}$
improves feasibility and appears both in the computation of the parameter $\theta_{k+1}^t$ (Step 3) and in the evaluation of $\operatorname{Pred}_k(\theta_k)$ and $\operatorname{Pred}_k(\theta_{k+1}^t)$ (Steps 3-4); however, it is not actually used to compute function or gradient estimates. On the other hand, the variables $y_{k+1}^t$ and $y_{k+1}^g$ are employed for computing three function estimates and one gradient estimate, respectively, but do not  necessarily enforce a decrease of the infeasibility measure.  
   
At Step 2   the function estimate $\bar f_k^{\dagger}=\bar{f}(x_k,y_{k+1}^t,\omega_{y_{k+1}^t}^{\dagger})$ and the gradient estimate $g_k=\nabla_{x} \bar{f}(x_k,y_{k+1}^g,\omega_{y_{k+1}^g})$ are computed for  noise realizations $\omega_{y_{k+1}^t}^{\dagger}$ and $\omega_{y_{k+1}^g}$, respectively. Then the linear model $m_k(p)=\bar f_k^{\dagger}+g_k^Tp$ is formed, and finally the search direction $p_k$ is computed  by minimizing the model over the ball of center zero and radius $\delta_k$, equivalently
\begin{equation*}
p_k=\underset{\|p\|\leq\delta_k}{\operatorname{argmin}} \ m_k(p)=-\delta_k\frac{g_k}{\|g_k\|}.
\end{equation*}

At Step 3, we compute the  trial penalty parameter $ \theta_{k+1}^t$ by using the predicted reduction $\Predk(\theta)$ defined as
\begin{equation}\label{eq:pred}
\Predk{}(\theta)=\theta(\bar{f}_k^*-m_k(p_k))+(1-\theta)(h(y_k)-h( \widetilde{y}_{k+1})),
\end{equation} 
where $\theta\in(0,1)$  and $\bar{f}^*_k=\bar{f}(x_k,y_{k+1}^t,\omega_{y_{k+1}^t}^*)$ is a further function estimate {computed for the noise realization $\omega_{y_{k+1}^t}^*$}. More precisely, the parameter $\theta_{k+1}^t$ is computed so that the following condition holds
\begin{equation}\label{eq:pred_cond_general}
\Predk(\theta)  \ge { \theta} \delta_k \|g_k\|.
\end{equation}
If condition (\ref{eq:pred_cond_general}) is satisfied with $\theta=\theta_k$, then we set ${ \theta_{k+1}^t}=\theta_k$. Otherwise
$\theta_{k+1}^t$ is computed as the largest value for which inequality (\ref{eq:pred_cond_general}) holds and its value given in (\ref{tkp1}) is  derived in  Lemma \ref{lem:thetak}.

At Step 4, we accept or reject the trial point $x_k+p_k$.  Given a function estimate $\bar{f}_k^p = \bar{f}(x_k+p_k,y_{k+1}^t,\omega_{y_{k+1}^t}^p)$, we define the actual reduction at the trial point $x_k+p_k$ as
\begin{equation}\label{eq:ared}
\Aredk( x_k+p_k,\theta) =
\theta (\bar{f}_k^* -   \bar{f}_k^p)+(1-\theta)(h(y_k)-h( y^t_{k+1})).
\end{equation} 
Then,  the iteration is declared successful, i.e., the point $x_k+p_k$ is accepted,   whenever the following three conditions hold
\begin{align}
\Aredk(x_k+p_k,\theta_{k+1}^t)&\geq \eta_1 \Predk(\theta_{k+1}^t),\label{eq:accept1}\\
\|g_k\|&\geq \eta_2\delta_k,\label{eq:accept2}\\
\theta_{k+1}^t & \geq \underline{\theta} \label{eq:accept3}.
\end{align}
Condition (\ref{eq:accept1}) mimics the classical acceptance criterion of standard trust-region methods, whereas \eqref{eq:accept2}-\eqref{eq:accept3} are technical conditions on the norm of the gradient estimator and the penalty parameter needed for the  
theoretical analysis of Algorithm \nomepunto. If conditions \eqref{eq:accept1}-\eqref{eq:accept3} are all satisfied, we update the iterates $x_{k+1},y_{k+1},\delta_{k+1},\theta_{k+1}$ as in (\ref{eq:succ}) and proceed to the next iteration, otherwise we retain the  iterate $x_k$, update the parameters as in (\ref{eq:unsucc}) and proceed to the next iteration.

 Condition (\ref{eq:accept3}) is not standard for deterministic Inexact Restoration approaches and deserves some comments. Such a condition has been introduced in order to ensure that the sequence $\{\theta_k\}$ is bounded away from zero, allowing us to avoid certain technical assumptions, employed in the previous literature, that may not be satisfied in our settings. More precisely, in seminal deterministic algorithms such as  \cite{BKM18, BKM20_1, BLM}, it is proved that the sequence $\{\theta_{k}\} $ is bounded away from zero by enforcing a condition on the noisy function values along the iterations. Such a condition is problem dependent and hard to impose unless the function $\bar f(x,y,\omega)$ is Lipschitz continuous with respect to $y$ and $\omega$, which is hardly verifiable or satisfied if the function evaluations are random, as is in our case. 
Likewise, in the stochastic Inexact Restoration algorithms  [4--6] for finite-sum minimization, it is proved that $\{\theta_{k}\}$ is bounded away from zero assuming  that every subsampled function evaluation is uniformly bounded from below and above on a set containing the generated sequence $\{x_k\}$. For the sake of generality, we preferred to abandon such an assumption and enforce condition \eqref{eq:accept3} instead.

We conclude this section by comparing \nome with existing trust-region algorithms with random models.
First, consider the algorithm {\sc storm} given in \cite{bcms, ChenMeniSche18}. On the one hand, Step 1 in \nome enforces sufficient accuracy  in probability for the function and gradient estimates under a suitable choice  of the parameter $r$ (see the upcoming Lemma \ref{lem_prob}).
On the other hand, setting the noise level is not explicitly part of {\sc storm}. Furthermore,  the actual reduction (\ref{eq:ared})   used in Step 4  is a convex combination of the reduction in the function approximation from $x_k$ to $x_k+p_k$ and the reduction in the function $h$. Thus, \nome may accept iterations where no decrease occurs in function approximations, provided that there has been a sufficient improvement towards the accuracy constraint $h(y)=0$. Such an acceptance rule differs from that 
 in {\sc storm}, which is solely based on  function estimates.

Second, we highlight some important differences between \nome and {\sc sirtr}, the stochastic trust-region algorithm with Inexact Restoration proposed in \cite{Bellavia2023}.  Algorithm {\sc sirtr} is applicable only to finite-sum minimization problems and its update rule for the accuracy variable is merely based on feasibility restoration, while not ensuring the probabilistic accuracy requirements \eqref{eq:prob_requirements} for random functions and gradients. Consequently, it may  be necessary to further enlarge the sample size originated by the update rule of {\sc sirtr} in order to guarantee the worst-case complexity result in expectation of the algorithm \cite{BMR}.
On the contrary, {\sc irerm} is applicable  to the general problem (\ref{eq:min}) and adjusts the accuracy variables according to the probabilistic requirements \eqref{eq:prob_requirements}. 

A further relevant difference between {\sc irerm} and {\sc sirtr} lies in the presence of the parameter $\underline \theta$ in {\sc irerm}, which allows us to avoid the  strong assumption,  made in {\sc sirtr}, that every subsampled function evaluation is uniformly bounded from below and above on a set containing the generated sequence $\{x_k\}$. As a result, the predicted reduction $\operatorname{Pred}_k(\theta)$ and actual reduction $\operatorname{Ared}_k(x_k+p_k,\theta)$ have been also modified, by inserting an additional function evaluation $\bar{f}_k^*$ that is not present in {\sc sirtr} and ensures that $\theta_{k+1}^t\geq \underline{\theta}$ whenever the function and gradient evaluations are sufficiently accurate (see the upcoming Lemma \ref{lemma_bart}).

\section{Analysis of the proposed algorithm} \label{S_dim}
In this section, we analyze the theoretical properties of the stochastic process underlying Algorithm \nomepunto.  
Following \cite{ChenMeniSche18}, we denote:    $X_k$   the random  iterate with realization $x_k$; 
$ \tilde Y_k$, the  random variable  
with realization $\tilde y_k$;  $G_k$ the random gradient with realization $g_k$;
$ \Delta_k$ the random trust-region radius with realization $\delta_k$; $P_k$ the random 
step  with realization $p_k$; $\bar{F}_k^{\dagger}$, $\bar{F}_k^{*}$, $\bar{F}_k^p$ the random estimates of $f$ with realizations
$\bar f_k^{\dagger}$, $\bar{f}_k^*$, $\bar{f}_k^p$. We denote with $\mathbb{P}_{k-1}(\cdot)$ and $\mathbb{E}_{k-1}[\cdot]$ 
the probability and expected value conditioned to the  $\sigma$-algebra generated 
up to the beginning  of iteration $k$, i.e., the   $\sigma$-algebra generated  by $\bar{F}_0^\dagger,\bar{F}_0^*, \bar F_0^p, \ldots,\bar{F}_{k-1}^\dagger,\bar{F}_{k-1}^*,\bar{F}_{k-1}^p$,
and $G_0, \ldots, G_{k-1}$, and with 
$\mathbb{P}_{k-1/2}(\cdot)$ 
the  probability conditioned to the  $\sigma$-algebra generated   by $\bar{F}_0^\dagger,\bar{F}_0^*, \bar F_0^p, \ldots,\bar{F}_{k-1}^\dagger,\bar{F}_{k-1}^*,\bar{F}_{k-1}^p$,
and $G_0, \ldots, G_{k}$.

In order to show that our algorithm can reach a desired accuracy in the value of the true gradient
of $f$ we give an upper bound on the expected value of the hitting time defined below.
\vskip 5pt
\begin{definition}\label{Nepsilon}
Given $\epsilon>0$,  the hitting time $\mathcal{K}_{\epsilon}$ is the random variable
\begin{equation*}
	\mathcal{K}_{\epsilon}=\min\{k\geq 0: \ \|\nabla f(X_k)\|\leq \epsilon\},
\end{equation*}
i.e., $\mathcal{K}_{\epsilon}$ is the first iteration such that $\|\nabla f(X_{\mathcal{K}_{\epsilon}})\|\le \epsilon$.
\end{definition}
\vskip 5pt
 
We proceed as follows. First, we show some properties of the sequence $\{\theta_k\}$; second, we introduce conditions  which enforce successful iterations; third, we provide results
that allow us to rely on the theory given in \cite{bcms} and to derive the iteration complexity results.

\subsection{On the sequence $\{\theta_k\}$}
We study the properties of the sequence $\{\theta_k\}$ that will be crucial in our analysis.
 
\vskip 5pt
\begin{lemma}\label{lem:thetak}
Suppose Assumption \ref{ass_fh} holds. The following facts hold true.
\begin{itemize}
\item[(i)] The sequence $\{\theta_k\}$ is positive, nonincreasing, and $\theta_{k+1}^t\le \theta_k$, for all $k\geq 0$.
\item[(ii)] It holds 
\begin{equation}\label{eq:pred_cond}
\Predk(\theta_{k+1}^t)\geq \theta_{k+1}^t\delta_k\|g_k\|, \quad \forall \ k\geq 0.
\end{equation}
\item[(iii)] The sequence $\{\theta_k\}$ is bounded away from zero with $\theta_{k+1}\geq \underline{\theta}>0$, for all $k\geq 0$. 
\end{itemize} 
\end{lemma}

\begin{proof}
(i) We note that $\theta_0>0$, and proceed by induction assuming that $\theta_k>0$ for some $k\geq 0$. First, suppose that $h(y_k)=0$. By (\ref{updateh}), we have that $h(y_{k+1}^t)=h(\widetilde{y}_{k+1})=0$, {which by Assumption \ref{ass_fh}(i)} yields $\bar{f}_k^*= \bar f_k^{\, \dagger} =f(x_k)$. Thus, 
$\Predk(\theta_k)  = { \theta_k} \delta_k \|g_k\|$, the update rule \eqref{tkp1} 
gives $\theta_{k+1}^t=\theta_k$, and  $\theta_{k+1}^t$ is positive by induction.
Now suppose that $h(y_k)>0$. If inequality $\Predk(\theta_k) \ge  \theta_k \delta_k\|g_k\|$ holds, 
then the update rule \eqref{tkp1} gives $\theta_{k+1}^t=\theta_k$, and thus $\theta_{k+1}^t>0$ by induction. Otherwise, it must be
\begin{equation}\label{tk}
\theta_k  \left( \bar f_k^*-\bar f_k^{\, \dagger} -(h(y_k)-h( \widetilde{y}_{k+1} ))\right)   < -(h(y_k)-h( \widetilde{y}_{k+1}))<0,
\end{equation}
where the rightmost inequality follows from (\ref{updateh}). Since $\theta_k>0$, \eqref{tk} implies
\begin{equation}\label{eq:den_neg}
 \bar f_k^*-\bar f_k^{\, \dagger} -(h(y_k)-h( \widetilde{y}_{k+1} )<0
\end{equation}
and
\begin{equation}\label{eq:thetakk1}
0<\theta_{k+1}^*:= 
\frac{h(y_k)-h(\widetilde{y}_{k+1})}{\bar f_k^{\, \dagger} -\bar f_k^* +h(y_k)-h( \widetilde{y}_{k+1})} <\theta_k. 
\end{equation}
By the update rule \eqref{tkp1}, we have $\theta_{k+1}^t=\theta_{k+1}^*$, which is positive from \eqref{eq:thetakk1}. Since Step 4 assigns either $\theta_{k+1}=\theta_{k+1}^t$ or $\theta_{k+1}=\theta_k$, we conclude that $\theta_{k+1}$ is positive in both cases. Hence, the sequence $\{\theta_k\}$ is positive. Moreover, from \eqref{tkp1} and \eqref{eq:thetakk1}  we have $\theta_{k+1}^t\leq \theta_k$, and $\{\theta_k\}_{k\in\N}$ is nonincreasing by Step 4.

(ii) If inequality $\Predk(\theta_k) \ge  \theta_k \delta_k\|g_k\|$ holds, 
the update rule \eqref{tkp1} gives $\theta_{k+1}^t=\theta_k$, and thus \eqref{eq:pred_cond} trivially holds. Otherwise, the update rule \eqref{tkp1} gives $\theta_{k+1}^t=\theta_{k+1}^*$, which satisfies (\ref{eq:pred_cond}). Indeed, imposing 
 $\Predk(\theta )\ge { \theta} \delta_k\|g_k\|$ yields
$$
\theta  \left( \bar f_k^*-\bar f_k^{\, \dagger} -(h(y_k)-h( \widetilde{y}_{k+1} ))\right)   \ge  -(h(y_k)-h( \widetilde{y}_{k+1})).
$$
From \eqref{eq:den_neg}, it follows that $\theta$ must satisfy 
$$
\theta\le \theta_{k+1}^*= 
\frac{h(y_k)-h(\widetilde{y}_{k+1})}{\bar f_k^{\, \dagger} -\bar f_k^* +h(y_k)-h( \widetilde{y}_{k+1})},
$$
hence condition \eqref{eq:pred_cond} holds.

(iii) Note that $\theta_0> \underline{\theta}$ by the initial setting and suppose by induction that  $\theta_{k}\geq \underline{\theta}$. By \eqref{eq:succ} and \eqref{eq:unsucc}, it holds that  $\theta_{k+1}=\theta_{k+1}^t$ for successful iterations,  $\theta_{k+1}=\theta_{k}$ otherwise. In the 
former case, we have $\theta_{k+1}\geq \underline \theta $
due to the definition of successful iteration, in the latter case 
$\theta_{k+1}\geq \underline \theta $ by induction. Thus, the claim follows. 
\end{proof}

\subsection{Successful iterations}
In this section, we introduce conditions that enforce successful iterations, i.e.,  conditions 
which guarantee (\ref{eq:accept1})--(\ref{eq:accept3}).
The key issue is the characterization of sufficient accuracy  on $f$ and $\nabla f$. In the following definition, we formalize such accuracy requirements and call {\em true} an iteration where they hold, {\em false} otherwise.

\begin{definition}\label{true}   For a given iteration $k\geq 0$, let $x_k,g_k,\delta_k,x_k+p_k,y_k,y_{k+1}^t,\widetilde{y}_{k+1}, y_{k+1}^g$ be as in Algorithm \nomepunto. Then, given $\kappa>0$, we say that iteration  $k$ is true when 
\begin{equation}\label{eq:true1}
\begin{aligned}
|f(x_k)-\bar{f}_k^*| &\le  \kappa\min\{ \delta_k^2, h(y_k)\},\\
|f(x_k)-\bar{f}_k^{\, \dagger}|&\le  \kappa\min\{ \delta_k^2, h(y_k)\}, 
\end{aligned}
\end{equation}
and 
\begin{align}
\|\nabla f(x_k)-g_k\|&\le  \kappa \delta_k  ,\label{eq:true3}\\
|f(x_k+p_k)- \bar{f}_k^p|&\le  \kappa \delta_k^2.  \label{eq:true2}
\end{align}
\end{definition}
\noindent
Based on Section \ref{S_PF}, 
 we remark that true iterations occur with a certain probability if the level of noise is sufficiently small.  In the case of finite-sum minimization, iteration $k$  is true in probability  provided that the sample sizes used for computing $\bar f_k^*,\, \bar f_k^\dagger, \, g_k, \,  \bar f_k^p$ satisfy lower bounds of the form \eqref{eq:y};
 in the remaining cases of Section \ref{S_PF}, iteration $k$  is true in probability for specified  numbers of trials employed in the evaluation of $\bar f_k^*,\, \bar f_k^\dagger, \, g_k, \,  \bar f_k^p$ and complying with inequalities of the form \eqref{eq:py}.  By using Taylor expansion and assuming  Lipschitz continuity of $\nabla f$, the second condition in  (\ref{eq:true1}) and (\ref{eq:true3}) imply that model $m_k(p)= \bar f_k^{\dagger} +g_k^Tp$ is fully linear in a ball of center $x_k$ and radius $\delta_k$,  according to the definition of fully linear model\footnote{
Suppose $\nabla f$ is Lipschitz continuous and let $\kappa_{ef},\kappa_{eg}$ be some positive constants. A function $m_k$ is a $(\kappa_{ef},\kappa_{eg})$-fully linear model on $B=\{x: \|x-x_k\|\le \delta_k\}$ if $\forall y\in B$ 
$$\|\nabla f(y)-\nabla m_k(y)\|\le \kappa_{eg}\delta_k \ \ \mbox{ and }  \ \
|f(y)-m_k(y)|\le \kappa_{ef}\delta_k^2.$$
See e.g., \cite[Definition 3.1]{ChenMeniSche18}.
}.
 
Our first result concerns  condition (\ref{eq:accept3}).
\vskip 5pt
\begin{lemma}\label{lemma_bart}
Suppose Assumption \ref{ass_fh} holds. Let $k$ be an iteration  such that condition (\ref{eq:true1}) is satisfied with $ \kappa\leq (1-r-\underline{\theta})/(2\underline{\theta})$.  
Then, condition (\ref{eq:accept3}) holds.
\end{lemma}
\begin{proof}
Consider Step 3 of Algorithm \nomepunto. If $\theta_{k+1}^t=\theta_k$ in (\ref{tkp1}), 
Lemma \ref{lem:thetak}(iii) establishes that each $\theta_k$ is bounded below by $\underline{\theta}$ and then 
(\ref{eq:accept3}) holds. 
Otherwise,  by  the proof of Lemma \ref{lem:thetak}(i), we know that $h(y_k)>0$ and
$\bar{f}_k^{\, \dagger}-\bar{f}_k^* +h(y_k)-h( \widetilde{y}_{k+1}) >0$. Then, condition \eqref{eq:true1}  yields
$$
|\bar{f}_k^{\, \dagger}-\bar{f}_k^*|=|\bar{f}_k^{\, \dagger}-f(x_k)+ f(x_k)-\bar{f}_k^*|\le 2\kappa h(y_k) ,
$$
\begin{equation*}
0<\bar{f}_k^{\, \dagger}-\bar{f}_k^* +h(y_k)-h(\widetilde{y}_{k+1}) \le |\bar{f}_k^{\, \dagger}-\bar{f}_k^*|+ h(y_k)\le (1+2\kappa)h(y_k).
\end{equation*}
Then, by using  \eqref{tkp1} and (\ref{updateh}), we get
\begin{align*}
 \theta_{k+1}^t=\frac{h(y_k)-h(\widetilde{y}_{k+1})}{{\bar{f}_k^{\, \dagger}-\bar{f}_k^*} +h(y_k)-h(\widetilde{y}_{k+1})} 
\ge \frac{(1-r)h(y_k)}{(1+2\kappa)h(y_k)}
\ge \underline \theta,
\end{align*}
where the last inequality follows from the assumptions made on $\kappa$.
\end{proof}
\vskip 5pt
The next lemma mimics a standard result holding for deterministic trust-region methods; if 
the iteration is true and the trust-region is sufficiently small then the iteration is successful. We need the 
subsequent assumption on $f$.
\vskip 5pt
\begin{assumption}\label{ass_gradf}
The gradient of the function $f$ in (\ref{eq:min}) is Lipschitz continuous with constant $2L$.
\end{assumption}

\vskip 5pt
\begin{lemma}\label{lemma_succ}
Suppose Assumptions  \ref{ass_fh} and \ref{ass_gradf} hold, and that  iteration $k$ is true 
with $ \kappa\leq (1-r-\underline{\theta})/(2\underline{\theta})$. 
If the trust-region radius $\delta_k$ satisfies
\begin{equation}\label{eq:delta_succ}
\delta_k\leq \min\left\{   \frac{1}{\eta_2}\|g_k\|,\ \frac{\underline{\theta}(1-\eta_1)}{ \theta_0(3\kappa+L)+\mu} \|g_k\| \right\},
\end{equation}
then the iteration is successful.
\end{lemma}
\begin{proof}
Iteration $k$ is successful when conditions (\ref{eq:accept1})--(\ref{eq:accept3}) are satisfied. Since $k$ is true, Lemma \ref{lemma_bart} guarantees the fulfillment of (\ref{eq:accept3}), thus we need to show that 
(\ref{eq:accept1}) and (\ref{eq:accept2}) hold when $\delta_k$ satisfies \eqref{eq:delta_succ}. 

Trivially   (\ref{eq:delta_succ}) implies (\ref{eq:accept2}). Concerning (\ref{eq:accept1}), we combine the definition of $\Predk$ in \eqref{eq:pred}, the definition of $\Aredk$ in \eqref{eq:ared}, inequality \eqref{eq:pred_cond},  condition \eqref{updateh}  so as to obtain
\begin{align}
\Aredk(x_k+p_k,\theta_{k+1}^t)&-\eta_1\Predk(\theta_{k+1}^t)\nonumber\\
&=(1-\eta_1)\Predk(\theta_{k+1}^t)+\Aredk(x_k+p_k,\theta_{k+1}^t)
-\Predk(\theta_{k+1}^t)\nonumber\\
&=(1-\eta_1)\Predk(\theta_{k+1}^t)+ \theta_{k+1}^t(m_k(p_k)- \bar{f}_k^p)\nonumber\\
& \quad  +(1-\theta_{k+1}^t) (h( \widetilde{y}_{k+1})-h(y_{k+1}^t) )\nonumber \\
&\geq (1-\eta_1){\theta_{k+1}^t}\delta_k\|g_k\|+ \theta_{k+1}^t(m_k(p_k)- \bar{f}_k^p) \nonumber \\
& \quad  - (1-\theta_{k+1}^t)  h( y_{k+1}^t) \nonumber \\
&\geq (1-\eta_1){\theta_{k+1}^t}\delta_k\|g_k\|-\theta_{k+1}^t|m_k(p_k)-\bar{f}_k^p|\nonumber \\
& \quad - h(y_{k+1}^t)
\nonumber \\
&\geq (1-\eta_1){\theta_{k+1}^t}\delta_k\|g_k\|-\theta_{k+1}^t|m_k(p_k)-\bar{f}_k^p|- \mu \delta_k^2.\label{eq:ine_success}
\end{align}
The occurrence that $k$ is true, Assumption \ref{ass_gradf} and  $\theta_{k+1}^t\le \theta_k\leq \theta_0$ (see Lemma \ref{lem:thetak}(i)) give
\begin{align}
\theta_{k+1}^t|m_k(p_k)&-\bar{f}_k^p|\nonumber\\
&=\theta_{k+1}^t\left |\bar{f}_k^{\, \dagger}-f(x_k)+f(x_k+p_k)- \bar{f}_k^p +(g_k-\nabla f(x_k))^Tp_k \right.\nonumber\\ 
&\hspace*{35pt}  \left. +f(x_k)- f(x_k+p_k)+\nabla f(x_k)^T p_k\right|\nonumber\\
& \le  \theta_0 \left(|\bar{f}_k^{\, \dagger}-f(x_k)|+|f(x_k+p_k)- \bar{f}_k^p| \right)\nonumber \\ 
&+ \theta_0\left( \|g_k-\nabla f(x_k)\|\delta_k+ \left |\int_0^1\left(\nabla f(x_k)-\nabla f(x_k+\tau p_k)\right)^Tp_kd\tau\right|\right) \nonumber \\
&\le \theta_{0} \left(3\kappa  +L\right)\delta_k^2.
\label{eq:bound3}
\end{align}
By plugging \eqref{eq:bound3} and (\ref{eq:accept3}) in \eqref{eq:ine_success}, it follows
\begin{align*}
\Aredk(x_k+p_k,\theta_{k+1})-\eta_1\Predk(\theta_{k+1})
&\geq \left( \underline{\theta}(1-\eta_1)\|g_k\|-
\left(\theta_0\left(3\kappa +L \right)+\mu \right) \delta_k\right)\delta_k,
\end{align*}
and  \eqref{eq:delta_succ} guarantees that the acceptance condition (\ref{eq:accept1}) is satisfied, since the right-hand side of the   inequality above is non-negative. 
\end{proof}
\vskip 5pt
We conclude this section by providing a straightforward lower bound on $\Aredk(x_{k+1},\theta_{k+1})$ in case $k$ is a successful iteration.
\begin{lemma}\label{lem:ared} Suppose Assumption \ref{ass_fh} holds. If the iteration $k$ is successful, then we have
\begin{equation}\label{eq:ared_successfulbis}
\Aredk(x_{k+1},\theta_{k+1})=
\Aredk(x_k+p_k,\theta_{k+1})\geq \eta_1 \delta_k\theta_{k+1}\|g_k\|.
\end{equation}
\end{lemma}
\begin{proof}
If the iteration $k$ is successful, then $x_{k+1}=x_k+p_k$ and $\theta_{k+1}=\theta_{k+1}^t $ according to \eqref{eq:succ}. By combining \eqref{eq:pred_cond} with the acceptance conditions \eqref{eq:accept1} and \eqref{eq:accept2}, it follows that
\begin{equation*}
\Aredk(x_k+p_k,\theta_{k+1})\geq  \eta_1 \Predk(\theta_{k+1}) \geq 
\eta_1{\theta_{k+1}} \delta_{k}\|g_k\|. 
\end{equation*}
\end{proof}

\subsection{Iteration complexity}
We conclude our analysis by deriving the expected number $\expect[{\mathcal{K}}_\epsilon]$  for the hitting time given in Definition \ref{Nepsilon}.
We rely on results for random processes introduced in \cite[\S 2]{bcms} and, to this end, we need some technical results
that are proved in this section.

Let $v\in(0,1)$ be a prefixed constant to be later specified and $\Sigma$ a constant such that
\begin{equation}\label{eq:sigma}
{ f(x_k)}-h(y_k)+\Sigma \geq 0, \quad \forall  \ k\ge 0.
\end{equation}
By Assumption \ref{ass_fh}, such a constant exists since $ f(x)$ is bounded from below and the values of $h(y_k)$ 
are upper bounded by $h_{up}$; one possible choice is $\Sigma= -f_{low}+h_{up}$.

We define the Lyapunov function
\begin{equation}\label{eq:Lyapunov}
\psi(x,y,\theta,\delta)=v(\theta {f(x)}+(1-\theta)h(y)+\theta\Sigma ) + (1-v)\delta^2,
\end{equation}
and consider the sequence $\{\psi_k\}$ where
\begin{equation}\label{eq:phik}
\psi_k=\psi(x_k,y_k,\theta_k,\delta_k), \quad \forall \ k\geq 0.
\end{equation}

Further, we introduce four events that characterize true iterations.
\begin{definition}\label{def_prob}
Given  $k\geq 0$,  let  $\mathcal{G}_{k,1}, \mathcal{G}_{k,2}, \mathcal{G}_{k,3}, \mathcal{G}_{k,4}$ be the events
\begin{align}
\mathcal{G}_{k,1}&=\left\{|f(X_k)- \bar{F}_k^{\dagger}|\leq  \kappa\min\{\Delta_k^2,h(Y_k)\}\right\},\label{eq:Gk1}\\
\mathcal{G}_{k,2}&=\left\{|f(X_k)-\bar{F}_k^*|\leq  \kappa\min\{\Delta_k^2,h(Y_k)\}\right\},\label{eq:Gk2}\\
\mathcal{G}_{k,3}&=\left\{\|\nabla f(X_k)-G_k\|\leq \kappa \Delta_k\right\},\label{eq:Gk3new}\\
\mathcal{G}_{k,4}&=\left\{|f(X_k+P_k)-\bar{F}_k^p|\leq \kappa\Delta_k^2\right\},\label{eq:Gk4new}
\end{align}
and $\mathcal{I}_k,\mathcal{J}_k$ the indicator functions
$${\mathcal{I}}_k=\indic({\mathcal{G}_{k,1}\cap \mathcal{G}_{k,2}\cap \mathcal{G}_{k,3}}), \quad {\mathcal{J}}_k=\indic({\mathcal{G}_{k,4}}).$$
\end{definition}
\vskip 5pt 
By 
Definition \ref{true}, we note  that iteration $k$ is true  when the realization of $\mathcal{I}_{k}\mathcal{J}_k$ equals 1. 
\vskip 5pt
\begin{assumption}\label{ass_prob}
Given  $k\geq 0$,  the events $\mathcal{G}_{k,1}$, $\mathcal{G}_{k,2}$, $\mathcal{G}_{k,3}$ are independent.
\end{assumption}
Note that Assumption \ref{ass_prob} is satisfied whenever the random variables $\omega^\dagger_{y_{k+1}^t}$, $\omega^*_{y_{k+1}^t}$ and $\omega_{y_{k+1}^g}$ 
that model the noise are independent.

In the next lemma, we state that the iteration $k$ is true in probability, provided that Assumptions \ref{ass_fh} and \ref{ass_prob} hold and the parameter $\mu$ is   properly chosen. 

\vskip 5pt
\begin{lemma}\label{lem_prob} 
Suppose that Assumptions  \ref{ass_fh} and \ref{ass_prob} hold.
Given a probability $\pi\in (0,1)$, 
let $\bar{\mu}(\pi^{\frac{1}{4}})>0$ be defined as in Assumption \ref{ass_fh}(ii) and assume that the parameter $\mu$ in Algorithm \nome is such that $\mu\leq \kappa \bar{\mu}(\pi^{\frac{1}{4}})$ with $\kappa$ as in Definition \ref{true}. Then we have
\begin{equation}\label{probpi}
\Pbk({\mathcal{I}}_k{\mathcal{J}}_k=1)\geq \pi.
\end{equation}
\end{lemma}
\begin{proof}
From the update rule \eqref{updateh} and condition $\mu \leq  \kappa \bar{\mu}(\pi^{\frac{1}{4}})$, we have
\begin{equation*}  
h({y}^t_{k+1})\leq \bar{\mu}(\pi^{\frac{1}{4}})\min\left\{\kappa \delta_k^2,\kappa h(y_k)\right\}, \,\, h({y}^g_{k+1})\leq \bar{\mu}(\pi^{\frac{1}{4}})(\kappa \delta_k).
\end{equation*}
Then, by Assumption \ref{ass_fh}(ii), it follows that
$$
\Pbk(\mathcal{G}_{k,i})\geq \pi^{\frac{1}{4}}, \quad 1\leq i\leq 3.
$$
Since the events $\mathcal{G}_{k,i}$ are independent by Assumption \ref{ass_prob}, the above inequality yields
$$
\Pbk(\mathcal{I}_k=1)\geq \pi^{\frac{3}{4}}.
$$
Analogously, from the update rule \eqref{updateh} and Assumption \ref{ass_fh}(ii), we have
$$
\Pbkd(\mathcal{J}_k=1)\geq \pi^{\frac{1}{4}}.
$$
Then, the thesis follows as in \cite[p. 100]{bcms}.
\end{proof}

\vskip 5pt

In the rest of the analysis, we will make use of a series of parameters  introduced below which are positive by construction.
\vskip 5pt
\begin{definition}\label{costanti}
Let $\kappa$, $\zeta$ be some scalars satisfying
\begin{align}
&0<\kappa<  \min\left\{\frac{1-r-\underline{\theta}}{2\underline{\theta}},\frac{\eta_1\eta_2}{2}\right\}\label{eq:kappa_condition},\\
& \zeta> \kappa+\max\left\{\eta_2,\frac{\theta_0(3\kappa+L){ +\kappa\bar{\mu}_{up}}}{\underline{\theta}(1-\eta_1)},\frac{\kappa(2+\eta_1)}{\eta_1}\right\}.\label{eq:zeta_condition}
\end{align}
Let $C_1,\, C_2,\, C_3,\, C_4$ be the positive constants
\begin{equation}
\begin{cases}
C_1 =  \underline{\theta}\left(\eta_1\left(1-\frac{\kappa}{\zeta}\right)-\frac{2\kappa}{\zeta}\right),\\
C_2 = 1+ \displaystyle\frac{L+\kappa\bar{\mu}_{up}}{\zeta},\\
C_3 = \underline{\theta}(\eta_1\eta_2-2\kappa),\\
C_4 = L+\zeta+\kappa\bar{\mu}_{up},
\end{cases}
\label{eq:C1}
\end{equation}
and let $v\in (0,1)$ be the scalar such that
\begin{equation}
\frac{v}{1-v}= \max\left\{\frac{4\gamma^2}{\zeta C_1},\frac{2\gamma^2}{C_3}\right\} \label{eq:v_condition}. 
\end{equation}
\end{definition}
\vskip 5pt\noindent 
 We note that the constants $C_1,\,C_3$ are positive due to \eqref{eq:zeta_condition} and \eqref{eq:kappa_condition}, respectively.

Letting $\Psi_k$ be the random function with realizations $\psi_k$, the following theorem shows a bound on the expected value $\Ek[\Psi_{k+1}-\Psi_k]$, from which one easily derives that $\Delta_k$ tends to zero  almost surely.
\vskip 5pt
\begin{theorem}\label{teo_limidelta0}
Let $\kappa,\, \zeta, \,C_1,\, C_2,\, C_3,\, C_4$ and $v$ be as in Definition \ref{costanti}. Let   
$\pi\in\left (\frac{1}{2},1\right )$ be a probability satisfying
\begin{align}
\displaystyle&\frac{\pi-\frac{1}{2}}{1-\pi}\geq \frac{C_2}{C_1}, \quad 1-\pi\leq \frac{\gamma^2-1}{2(\gamma^4-1)+2\gamma^2C_4\max\left\{\frac{4\gamma^2}{\zeta C_1},\frac{2\gamma^2}{C_3}\right\} }. \label{eq:pi_condition}
\end{align}
Suppose that Assumptions \ref{ass_fh}, \ref{ass_gradf} and \ref{ass_prob} hold. Let $\bar{\mu}(\pi^{\frac{1}{4}})>0$ be defined as in Assumption \ref{ass_fh}(ii) and assume $\mu\leq \kappa \bar{\mu}(\pi^{\frac{1}{4}})$ with $\kappa$ as in Definition \ref{true}. Let $\psi_k$ be defined as in (\ref{eq:phik}) and (\ref{eq:sigma}). Then, there exists $\sigma>0$ such that
\begin{equation}\label{eq:suff_decrease}
{\Ek[\Psi_{k+1}-\Psi_k]\leq -\sigma \Delta_k^2, \quad \forall \ k\geq 0.}
\end{equation}
Furthermore, 
\begin{equation}\label{eq:delta_summable}
\sum_{k=0}^{\infty}\Delta_k^2<\infty
\end{equation}
almost surely.
\end{theorem}

\begin{proof}
See the Appendix.
\end{proof}
\vskip 5pt
We observe that (\ref{eq:delta_summable}) implies that $\Delta_k$ tends to zero almost surely.  Thus, the accuracy requirements given in Definition \ref{true} for true iterations   are expected to become more stringent as the iterations proceed.  

In order to obtain the expected value for the hitting time $\mathcal{K}_{\epsilon}$ given in Definition \ref{Nepsilon}, 
we need a value $\delta^\dagger$, independent of $k$, such that if $\delta_k<\delta^\dagger$, 
iteration $k$ is true and $k<\mathcal{K}_{\epsilon}$,
then the iteration is successful. We specify the value of such $\delta^{\dagger}$ in the next Lemma.
\vskip 5pt

\begin{lemma}\label{ddagger}
Let Assumptions \ref{ass_fh}, {\ref{ass_gradf}} hold. 
Suppose that  $\|\nabla f(x_k)\|> \epsilon$ for some $\epsilon>0$, 
the iteration $k$ is true,  and 
\begin{equation} \label{def_b}
\delta_k<\delta^\dagger:= \min \left \{ \frac{\epsilon}{2\kappa}, \frac{\epsilon}{2\eta_2}, 
\frac{ \underline{\theta}\epsilon(1-\eta_1)}{2(\theta_0(3\kappa+L)+\mu)}\right\}.
\end{equation}
Then, iteration k is successful.
\end{lemma}
\begin{proof}
Conditions (\ref{eq:true3}) and (\ref{def_b}) give 
$\|g_{k}-\nabla f(x_k)\|\le \kappa \delta_k <\frac{\epsilon}{2}$.
Consequently   $\|\nabla f(x_k)\|>\epsilon$ implies $\|g_k\|\ge \frac{\epsilon}{2}$, and 
iteration $k$ is successful by Lemma \ref{lemma_succ}.
\end{proof}
\vskip 5pt

Let  ${\mathcal{G}}_k=\mathcal{G}_{k,1}\cap \mathcal{G}_{k,2}\cap \mathcal{G}_{k,3}\cap \mathcal{G}_{k,4}$ with $\mathcal{G}_{k,1},\mathcal{G}_{k,2},\mathcal{G}_{k,3},\mathcal{G}_{k,4}$ given in Definition \ref{def_prob}, $\indic({\mathcal{G}}_k)$ be the value of the indicator function
for the event ${\mathcal{G}}_k$ and $W_k$  be the random variable  defined as
\begin{equation}\label{eq:Wk}
\begin{cases}
W_0 =1\\
W_{k+1}=2\left(\indic({\mathcal{G}}_k) -\frac{1}{2}\right), \quad k=0,1,\ldots
\end{cases}
\end{equation}
Clearly, $W_k$ takes values $\pm 1$.  Then, we can prove the following result.

\begin{lemma}\label{ass_bcs}
Let the assumptions of Theorem \ref{teo_limidelta0} hold,   $\delta^\dagger$ as in (\ref{def_b})
and  $\mathcal{K}_{\epsilon}$ as in Definition \ref{Nepsilon}. Suppose there exists some $j_{\max}\geq 0$ such that $\delta_{\max}=\gamma^{j_{\max}}\delta_0$, and $\delta_0>\delta^{\dagger}$. 
	Then,
\begin{itemize}[leftmargin = 18pt]
\item[i)] there exists $\lambda>0$ such that $\Delta_k\leq \delta_0e^{\lambda \cdot j_{\max}}$ for all $k\geq 0$;
\item[ii)] there exists a constant $ \delta_{\epsilon}=\delta_0e^{\lambda \cdot j_{\epsilon}}$ for some $j_{\epsilon}\leq 0$ such that, for all $k\geq 0$,
\begin{equation}\label{eq:ii}
\mathbbm{1}_{\{\mathcal{K}_{\epsilon}>k\}}\Delta_{k+1}\geq \mathbbm{1}_{\{\mathcal{K}_{\epsilon}>k\}}\min\{\Delta_ke^{\lambda W_{k+1}}, \delta_{\epsilon} \}.
\end{equation}
\item[iii)] there exists a nondecreasing function $\ell:[0,\infty)\rightarrow (0,\infty)$ and a constant $\Theta>0$ such that, for all $k\geq 0$,
\begin{equation}
	\mathbbm{1}_{\{\mathcal{K}_{\epsilon}>k\}}\mathbb{E}_{k-1}[\Psi_{k+1}]\leq \mathbbm{1}_{\{\mathcal{K}_{\epsilon}>k\}}\Psi_k-\mathbbm{1}_{\{\mathcal{K}_{\epsilon}>k\}}\Theta \ell(\Delta_k).
\end{equation}
\end{itemize}
\end{lemma}

\begin{proof}
The proof parallels that of \cite[Lemma 7]{bcms}.	
	
i) Since $\delta_{\max}=\gamma^{j_{\max}}\delta_0$, we can set $\lambda=\log(\gamma)>0$, and the thesis follows from 
imposing $\delta_k\le \delta_{\max}$, for all $ k\ge 0$, see Step 4 of Algorithm \nomepunto.

ii) Let us set
\begin{equation}\label{eq:delta_epsilon}
	{ \delta_\epsilon=\frac{\epsilon}{\xi}}, \quad \text{where }\xi\geq \max\left\{2\kappa ,2\eta_2,
	\frac{2\theta_0(3\kappa+L)}{\underline{\theta}(1-\eta_1)}  \right\},
\end{equation}
and assume that $ \delta_{\epsilon}=\gamma^{j_{\epsilon}}\delta_0$, for some integer $j_{\epsilon}\leq 0$; 
notice that we can always choose $\xi$ sufficiently large so that this is true. 
As a consequence, $\Delta_k= \gamma^{i_k} \delta_\epsilon$ for some random integer variable $i_k$.

When $\mathbbm{1}_{\{\mathcal{K}_{\epsilon}>k\}}=0$, inequality \eqref{eq:ii} trivially holds. 
Otherwise, conditioning on $\mathbbm{1}_{\{\mathcal{K}_{\epsilon}>k\}}=1$, let us show that
\begin{equation}\label{eq:thesis}
\Delta_{k+1}\geq \min\{\delta_\epsilon,\min\{\delta_{\max},\gamma\Delta_k\}\indic({\mathcal{G}}_k)+\gamma^{-1}\Delta_k(1-\indic({\mathcal{G}}_k))\}.
\end{equation}
Indeed,   taking into account the above
equation $\Delta_k=\gamma^{i_k}\delta_\epsilon$,  for any realization such that $\delta_k>\delta_\epsilon$ we have $\delta_k\geq \gamma\delta_{\epsilon}$, and 
the updating rule for $\delta_k$ in Step 4 of Algorithm \nome implies  $\delta_{k+1}\geq \delta_{\epsilon}$. Now let us consider a realization such that  $\delta_k\leq \delta_{\epsilon}$. Since $\mathcal{K}_{\epsilon}>k$ and $\delta_{\epsilon}\leq \delta^{\dagger}$, if $\indic({\mathcal{G}}_k)=1$ then we can apply Lemma \ref{ddagger} and conclude that $k$ is successful. Hence, by Step 4, we have $\delta_{k+1}=\min\{\delta_{\max},\gamma\delta_k\}$. If $\indic({\mathcal{G}}_k)=0$, then we cannot guarantee that $k$ is successful; however, again using
the updating rule for $\delta_k$ in Step 4 of Algorithm \nomepunto, we can write $\delta_{k+1}\geq \gamma^{-1}\delta_k$. Combining these two cases, we get \eqref{eq:thesis}. If we observe that $\delta_{\max}=\gamma^{j_{\max}}\delta_0\geq { \gamma^{j_{\epsilon}}\delta_{0}= \delta_{\epsilon}}$, and recall the definition of $W_k$ in \eqref{eq:Wk}, then equation \eqref{eq:thesis} easily yields \eqref{eq:ii}.

(iii) The thesis trivially follows from \eqref{eq:suff_decrease} with $\ell(\Delta)=\Delta^2$ and $\Theta=\sigma$.
\end{proof}

We now state the main result.
\vskip 5pt
\begin{theorem}\label{teo_E}
	Under the assumptions of Lemma \ref{ass_bcs}, we have
	\begin{equation}{ 
		\mathbb{E}[\mathcal{K}_{\epsilon}]\leq \frac{{\pi}}{2{\pi}-1}\cdot \frac{\psi_0\xi^2}{{\sigma}  \epsilon^2}}+1.
	\end{equation}
	where $\sigma$ is the scalar in \eqref{eq:suff_decrease} 
	\end{theorem}
	
\begin{proof}
The claim coincides with \cite[Theorem 2]{bcms}, which holds for any  random process $\{(\Phi_k,\Delta_k,W_k)\}_{k\in\mathbb{N}}$ satisfying \cite[Assumption 2.1]{bcms}, i.e., conditions $i), ii), iii)$ stated in  Lemma \ref{ass_bcs}.  
In particular, following \cite{bcms},  let $A_n$ count the iterations for which $\Delta_k\ge \delta_\epsilon$ on  $\mathcal{K}_{\epsilon}>k$, and let $N(k)= \max\{n : A_n\le k\}$.   Lemma 2 in \cite{bcms} shows that 
$\mathbb{E}[N(\mathcal{K}_{\epsilon} - 1) + 1] \le \frac{\psi_0\xi^2}{\sigma \epsilon^2}$.
Then, the Wald's Identity \cite[Theorem 1]{bcms}  gives 
$\mathbb{E}[A_{N(\mathcal{K}_{\epsilon} - 1)+1}] = \frac{\pi}{2\pi-1}\mathbb{E}[N(\mathcal{K}_{\epsilon} - 1) + 1]$, and the claim follows by $A_{N(\mathcal{K}_{\epsilon} - 1)+1} \ge \mathcal{K}_{\epsilon} - 1$ and 
the technical Lemma 1 in \cite{bcms}. 
\end{proof}

An iteration complexity result analogous  to Theorem \ref{teo_E} can be found in \cite[Theorem 3]{Bellavia2023} for 
the algorithm {\sc sirtr} applicable to finite-sum minimization and  based on the combination of the Inexact Restoration framework and the stochastic first-order trust-region method. The probabilistic requirements in \cite{Bellavia2023} are imposed only on the 
stochastic gradients and involve bounds of order $\delta_k$, similarly to (\ref{eq:true3}),
whereas {\sc irerm} requires the more stringent conditions (\ref{eq:true1}) 
and (\ref{eq:true2}). On the other hand, the convergence analysis in \cite{Bellavia2023} is carried  out
under the assumption that the objective function $f$ is bounded from above on a set containing the sequence of the iterates, which is not required in our analysis. 

\section{Numerical experiments}\label{S_exps}
In this section, we present results from our numerical validation of {\sc irerm}.
Following \cite[\S 6.1]{ChenMeniSche18}, we address the solution of the least-squares problem
\begin{equation*}
\min_{x\in\R^n}f(x)=\sum_{i=1}^{m} (f_i(x))^2,
\end{equation*}
where each $f_i$, $i=1, \ldots, m$, is a smooth function and suppose that the function evaluations are corrupted by multiplicative noise, i.e.,
\begin{equation}\label{ls_noisy}
f(x,\xi)=\sum_{i=1}^{m} ((1+\xi_i)f_i(x))^2,
\end{equation}
 with $\xi_i$, $i=1,\ldots,m$, generated from a uniform random distribution on $[-\sigma, \sigma]$ for some $\sigma>0$. We consider 17 nonlinear least-squares problems taken from \cite{Luksan}, which represent the overall set of problems extracted from References n. $8,21,26, 36$ in \cite[Section 2]{Luksan}. 
The selected problems are reported in Table \ref{tab0}. Note that  the dimension of the unknown is set to $n=100$ for each problem.

Since $f$ has the form (\ref{eq:stochastic}), i.e., $f(x)=\expect_{\xi}[f(x,\xi)]$,  we 
apply {\sc irerm} as discussed in Section \ref{S_PF}. Thus we set the constraint function in (\ref{eq:min_vinc}) equal to $h(y)=\sqrt{y}$, 
where $p(y)=\lceil 1/y\rceil\in \mathbb{N}$ represents the number of samples for a given $y\in (0,1] $, and  evaluate function and derivative approximations as in 
(\ref{eq:approximations3}) and (\ref{eq:approximations2}), respectively.

We apply {\sc irerm} and a first-order version of {\sc storm} \cite{ChenMeniSche18} to the problems  in Table \ref{tab0}  and compare the obtained results. Our implementation of the two algorithms is as follows.
Firstly, {\sc irerm} and {\sc storm} are implemented on the basis of the theoretical prescriptions
on function and derivative approximations derived in this work and in \cite{ChenMeniSche18}, respectively;
the resulting algorithms are denoted as
{\sc irerm}\_v1 and {\sc storm}\_v1. Secondly, {\sc irerm} and {\sc storm} are implemented ignoring 
the theoretical prescriptions and following Algorithm 3 in \cite[p. 485]{ChenMeniSche18}; the resulting
algorithms are denoted as {\sc irerm}\_v2 and {\sc storm}\_v2.

Specifically, in  {\sc irerm}\_v1 we set  the number of samples as
$$p(y_{k+1}^t) =  \left\lceil \frac{1}{\mu^2 \min\{h(y_k)^2,\delta_k^4\}}\right\rceil,\quad p(y_{k+1}^g) =  \left\lceil \frac{1}{\mu^2\delta_k^2} \right\rceil,
$$
with $ \mu=0.99$, so that the inequalities (\ref{updateh}) at Step 1 of Algorithm {\sc irerm} are satisfied. As for {\sc storm}\_v1,
we implement Algorithm 1 in \cite[p. 454]{ChenMeniSche18} by computing the average function evaluations with sample size $\left\lceil \frac{1}{\delta_k^4} \right\rceil$ and the average gradient evaluation  with sample size $\left\lceil \frac{1}{\delta_k^2} \right\rceil$, according to the discussion in \cite[\S 5]{ChenMeniSche18}.

On the other hand, following  Algorithm 3 in \cite[p. 485]{ChenMeniSche18}, we adopt
$\max\left \{ 10+k,  \left\lceil  \frac{1}{\delta_k^2} \right\rceil \right \}$
as the   heuristic choice for the number of samples in {\sc storm}\_v2, and employ the same value for $p(y_{k+1}^t)$ and $p(y_{k+1}^g)$ in {\sc irerm}\_v2. Such a choice for the sample sizes allows for the worsening of the infeasibility measure, in the spirit of the IR framework, although it does not strictly comply with Step 1 of Algorithm {\sc irerm}.

In both versions of {\sc irerm}, the parameters are set as $y_0 = 1$, $\theta_0 = 0.9$, $\underline{\theta}=10^{-8}$, $\delta_0=1$, $\delta_{\max}= 10$, 
$\gamma=2$, $\eta=0.1$, $\eta_2=10^{-3}$. Likewise, both versions of {\sc storm} are implemented with $\delta_0=1$, $\delta_{\max}= 10$, $\gamma=2$, $\eta=0.1$, $\eta_2=10^{-3}$. Hence, all algorithms are equipped with the same parameter values, except for the additional parameters $y_0,\theta_0,\underline{\theta}$ required in {\sc irerm}. Note that the parameter values common to {\sc irerm} and {\sc storm} are the same chosen in the numerical experiments of \cite[\S 6]{ChenMeniSche18}.

For each problem and algorithm, we performed $10$ runs using  the level of noise $\sigma=10^{-1}$ as  in \cite[\S 6]{ChenMeniSche18} and the initial guesses given in \cite{Luksan}.
We adopt  the number of samples employed in the average function and gradient approximations as the measure of the cost of the algorithms; more precisely, we use a counter   {\tt Cost} that is increased, at each function or gradient evaluation, by the number of samples employed at such evaluation. Furthermore, we impose a computational budget equal to  $10^5(n+1)$  evaluations  for  {\sc irerm}\_v1 and {\sc storm}\_v1, and equal to $10^4(n+1)$  evaluations  for {\sc irerm}\_v2 and {\sc storm}\_v2. We observe that {\sc irerm} requires three function evaluations and one gradient evaluation, whereas {\sc storm} asks for two function evaluations and one gradient evaluation per iteration. For all algorithms, we stop the iterations when either the maximum computational budget is exceeded or a maximum number of $k_{\max}=500$ iterations is met. Note that the upper bound $k_{\max}$ was never reached in our numerical tests.
\begin{small}
\begin{table}[t!] 
\begin{center}
\begin{tabular}{cclcc}
Problem & Problem name & Source & $n$ & $m$\\ \hline
{\sc p1} & {\sc Chained Rosenbrock function} & \cite[Problem 2.1]{Luksan}\hspace*{5pt} & 100 & $2(n-1)$ \\
{\sc p2} & {\sc Chained Wood function} & \cite[Problem 2.2]{Luksan}\hspace*{5pt} & 100 & $3(n-2)$ \\
{\sc p3} & {\sc Chained Powell singular function} & \cite[Problem 2.3]{Luksan}\hspace*{5pt} & 100 & $2(n-2)$\\
{\sc p4} & {\sc Chained Gragg and Levy function} & \cite[Problem 2.4]{Luksan}\hspace*{5pt} & 100 & $5(n-2)/2$ \\
{\sc p5} & {\sc Generalized Broyden tridiagonal function} & \cite[Problem 2.5]{Luksan}\hspace*{5pt} & 100 & $n$ \\
{\sc p6} & {\sc Generalized Broyden banded function} & \cite[Problem 2.6]{Luksan}\hspace*{5pt} & 100 & $n$ \\
{\sc p7} & {\sc Chained Freudstein and Roth function} & \cite[Problem 2.7]{Luksan}\hspace*{5pt} & 100 & $2(n-1)$ \\
{\sc p8} & {\sc Toint quadratic merging problem} & \cite[Problem 2.9]{Luksan}\hspace*{5pt} & 100 & $3(n-2)$ \\
{\sc p9} & {\sc Chained exponential problem} & \cite[Problem 2.10]{Luksan}\hspace*{5pt} & 100 & $2(n-1)$ \\
{\sc p10} & -- & \cite[Problem 2.53]{Luksan}\hspace*{5pt} & 100 & $n$ \\
{\sc p11} & -- & \cite[Problem 2.71]{Luksan} & 100 & $n$\\
{\sc p12} & -- & \cite[Problem 2.72]{Luksan} & 100 & $n$\\
{\sc p13} & {\sc nondquar} & \cite[Problem 2.79]{Luksan} & 100 & $n$\\
{\sc p14} & {\sc sinquad} & \cite[Problem 2.81]{Luksan} & 100 & $n$\\
{\sc p15} & {\sc edensch} & \cite[Problem 2.82]{Luksan}\hspace*{5pt} & 100 & $3(n-1)$ \\
{\sc p16} & {\sc genhumps} & \cite[Problem 2.83]{Luksan}\hspace*{5pt} & 100 & $3(n-1)$ \\
{\sc p17} & {\sc errinros (modified)} & \cite[Problem 2.84]{Luksan}\hspace*{5pt} & 100 & $2(n-1)$ \\
\hline
\end{tabular}
\end{center}
\caption{Test problems description.}\label{tab0}  
\end{table}
\end{small}

We now summarize the numerical results reported in this section. In Table \ref{tab2}, we show the lowest values of $f$ achieved at termination over the $10$ runs by algorithms  {\sc irerm}\_v1,  {\sc storm}\_v1, {\sc irerm}\_v2,  {\sc storm}\_v2, whereas in Table \ref{tab1} we report the average function value of $f$ and the standard deviation reached  at termination over the 10 runs by all algorithms, respectively. Focusing on problems where both versions {\sc v1} and {\sc v2} of algorithms {\sc irerm} and {\sc storm} show significant differences at termination, in Figure \ref{fig1} (resp. Figure \ref{fig2}) we plot the best runs of {\sc irerm}\_v1 and  {\sc storm}\_v1 (resp. {\sc irerm}\_v2 and  {\sc storm}\_v2), i.e., the runs where the smallest value of the exact function $f$ is achieved at termination. For each problem, we display the value of the exact function $f$ versus the computational cost ({\tt Cost}), and the trust-region radius $\delta_k$ versus the computational cost ({\tt Cost}).

\begin{small}
\begin{table}[h!] 
\begin{center}
\begin{tabular}{c|cc|cc}
Problem& {\sc irerm}\_v1 & {\sc storm}\_v1 & {\sc irerm}\_v2 & {\sc storm}\_v2 \\
\hline
{\sc p1} &  4.93e+01      &   4.90e+01  &   4.73e+01   &  4.78e+01     \\
{\sc p2} &  1.94e+02     &  1.46e+02  &  1.59e+02    &  1.84e+02    \\
{\sc p3} &  4.15e$-$02     &   3.14e$-$02   &    1.17e$-$04    &  6.47e$-$03       \\
{\sc p4} &    1.27e+01      &  1.27e+01     &    1.26e+01    &   1.26e+01    \\
{\sc p5} &  4.18e$-$05       & 8.23e$-$05  &     8.69e$-$07   &  6.29e$-$06    \\ 
{\sc p6} &    2.73e$-$03     & 4.14e$-$03  &     1.84e$-$07    &  2.76e$-$05     \\
{\sc p7} &  6.01e+03     &   6.01e+03   &   6.00e+03   &   6.00e+03     \\
{\sc p8} &  2.17e+02     &   2.17e+02  &   2.16e+02   &  2.17e+02    \\
{\sc p9} &    1.96e+01    & 1.96e+01  &   1.91e+01  &    1.91e+01   \\
{\sc p10} &   2.63e$-$08     &  7.12e$-$08 & 9.48e$-$08   &   1.40e$-$08    \\
{\sc p11} &   4.68e$-$06     & 1.54e$-$05  &  2.61e$-$08   &   4.83e$-$07    \\
{\sc p12} &   1.29e$-$05     & 1.02e$-$04  &   4.54e$-$07   &  3.30e$-$06    \\
{\sc p13} &   2.96e$-$01    &  2.23e$-$01  & 1.84e$-$01   &   1.54e$-$01    \\
{\sc p14} &   1.16e$-$01     & 1.29e$-$01  &  7.86e$-$02   &   7.06e$-$02   \\
{\sc p15} &  2.94e+02     &   2.94e+02 &   2.94e+02  &  2.94e+02    \\
{\sc p16} &  1.25e+06      &   6.74e+05 &   1.84e+03   &  3.38e+03     \\
{\sc p17} &  3.92e+01      & 3.92e+01  &  3.91e+01   &   3.90e+01     \\
\hline
\end{tabular}
\end{center}
\caption{Lowest value of $f$ reached at termination over 10 runs for {\rm{\sc irerm}} and {\rm{\sc storm}}.}\label{tab2}
\end{table}
\end{small}

\begin{small}
\begin{table}[h!]
\begin{center}
\begin{tabular}{c|cc|cc}
Problem& {\sc irerm}\_v1 & {\sc storm}\_v1 & {\sc irerm}\_v2 & {\sc storm}\_v2 \\
\hline
{\sc p1}  & 4.93e+01$\pm$3.79e$-$02 & 4.92e+01$\pm$1.01e$-$01 & 4.78e+01$\pm$5.40e$-$01   &  4.85e+01$\pm$6.74e$-$01 \\
 {\sc p2} & 1.95e+02$\pm$1.19e+00 &  1.56e+02$\pm$5.73e+00 & 1.82e+02$\pm$1.07e+01    & 1.87e+02$\pm$1.03e+00   \\
{\sc p3} & 9.63e$-$02$\pm$7.01e$-$02   & 7.11e$-$02$\pm$2.34e$-$02  & 2.85e$-$02$\pm$3.01e$-$02  & 8.31e$-$03$\pm$1.45e$-$03 \\
{\sc p4}  & 1.30e+01$\pm$7.81e$-$01     & 1.27e+01$\pm$5.58e$-$03  & 1.27e+01$\pm$1.75e$-$01   & 1.26e+01$\pm$6.52e$-$03 \\ 
{\sc p5} & 2.06e$-$03$\pm$8.64e$-$04  & 1.58e$-$03$\pm$1.06e$-$03 & 1.81e$-$05$\pm$1.61e$-$05    & 2.59e$-$05$\pm$2.11e$-$05 \\ 
{\sc p6} &  6.76e$-$03$\pm$3.24e$-$03      & 1.03e$-$02$\pm$3.20e$-$03  &  5.45e$-$05$\pm$7.39e$-$05    & 1.89e$-$04$\pm$8.23e$-$05     \\
{\sc p7} &  6.01e+03$\pm$6.77e$-$02      & 6.01e+03$\pm$7.17e$-$02  &  6.00e+03$\pm$3.26e+00   & 6.01e+03$\pm$1.12e+00     \\
{\sc p8} &  4.15e+03$\pm$6.22e+03     & 2.17e+02$\pm$8.22e$-$03  & 2.46e+02$\pm$2.52e+01   &  2.17e+02$\pm$5.18e$-$03   \\
{\sc p9} & 1.97e+01$\pm$5.36e$-$02      & 1.97e+01$\pm$6.86e$-$02  & 1.92e+01$\pm$4.67e$-$02    & 1.91e+01$\pm$1.29e$-$02      \\
{\sc p10} & 2.80e$-$05$\pm$3.64e$-$05       &  1.32e$-$05$\pm$1.21e$-$05 & 9.92e$-$07$\pm$5.61e$-$07  &  1.87e$-$07$\pm$1.62e$-$07    \\
{\sc p11} & 6.85e$-$05$\pm$5.05e$-$05      & 3.01e$-$05$\pm$1.12e$-$05  &  1.86e$-$06$\pm$1.48e$-$06    & 9.16e$-$07$\pm$2.47e$-$07  \\
{\sc p12} & 1.65e$-$04$\pm$1.54e$-$04      & 3.05e$-$04$\pm$1.21e$-$04  & 8.94e$-$06$\pm$1.57e$-$05 & 8.35e$-$06$\pm$3.24e$-$06     \\
{\sc p13} &  3.23e$-$01$\pm$2.66e$-$02     & 3.03e$-$01$\pm$2.85e$-$02  & 1.90e$-$01$\pm$6.66e$-$03     & 1.59e$-$01$\pm$4.94e$-$03   \\
{\sc p14} &  1.25e$-$01$\pm$6.48e$-$03     & 1.30e$-$01$\pm$1.38e$-$04  & 8.08e$-$02$\pm$1.62e$-$03    & 7.28e$-$02$\pm$1.06e$-$03    \\
{\sc p15} & 2.94e+02$\pm$1.33e$-$02      &  2.94e+02$\pm$8.39e$-$03 & 2.94e+02$\pm$1.46e$-$05  &  2.94e+02$\pm$2.07e$-$02    \\
{\sc p16} &  1.27e+06$\pm$6.14e+03     & 1.13e+06$\pm$1.68e+05  & 2.04e+03$\pm$1.08e+02     &  3.88e+03$\pm$3.65e+02   \\
{\sc p17} &  5.13e+01$\pm$1.61e+01     & 3.93e+01$\pm$4.00e$-$02  &  3.97e+01$\pm$9.28e$-$01   & 3.91e+01$\pm$3.10e$-$02    \\
\hline
\end{tabular}
\end{center}
\caption{Average function value $f$ and standard deviation at termination over 10 runs for {\rm{\sc irerm}} and {\rm{\sc storm}}.}\label{tab1}
\end{table}
\end{small}

\begin{figure}[t!]
\begin{center}
\begin{tabular}{ccc}
\includegraphics[scale = 0.3]{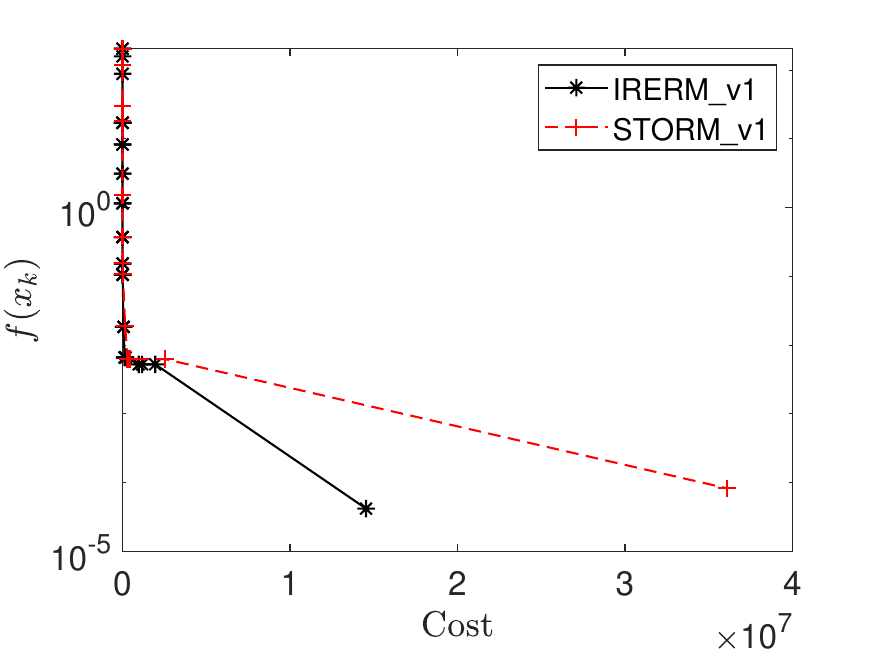}  & \includegraphics[scale = 0.3]{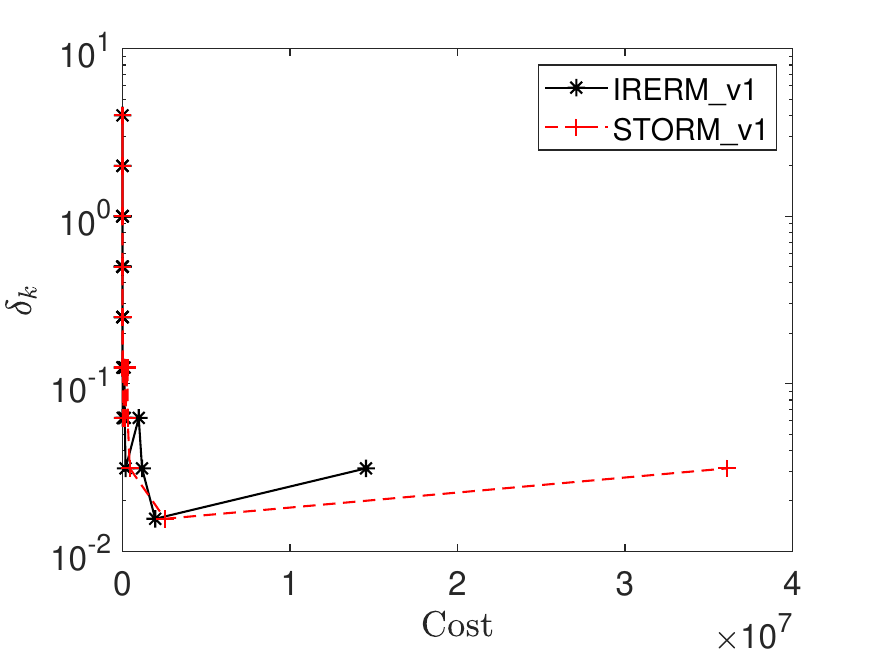}\\
\includegraphics[scale = 0.3]{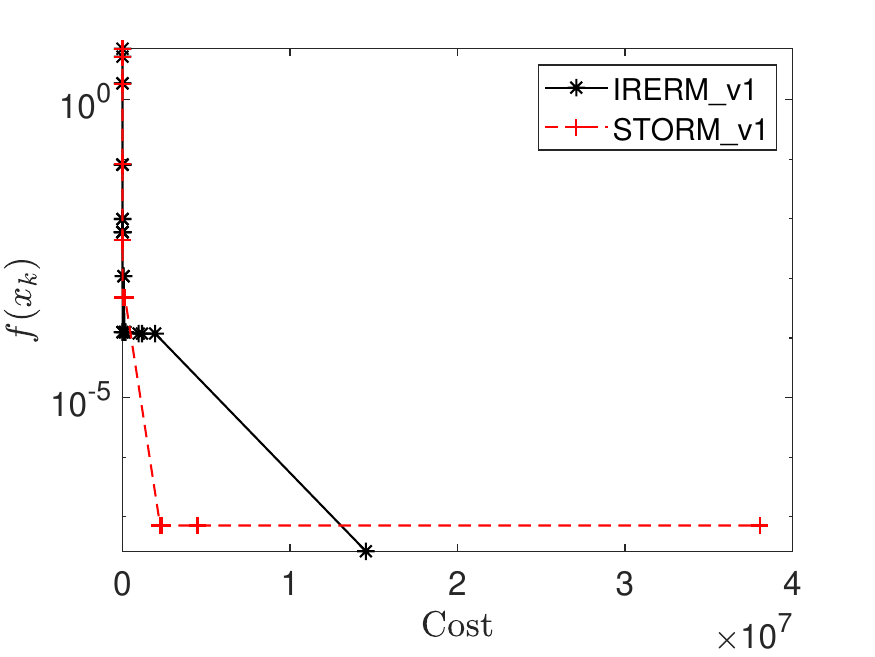}  & \includegraphics[scale = 0.3]{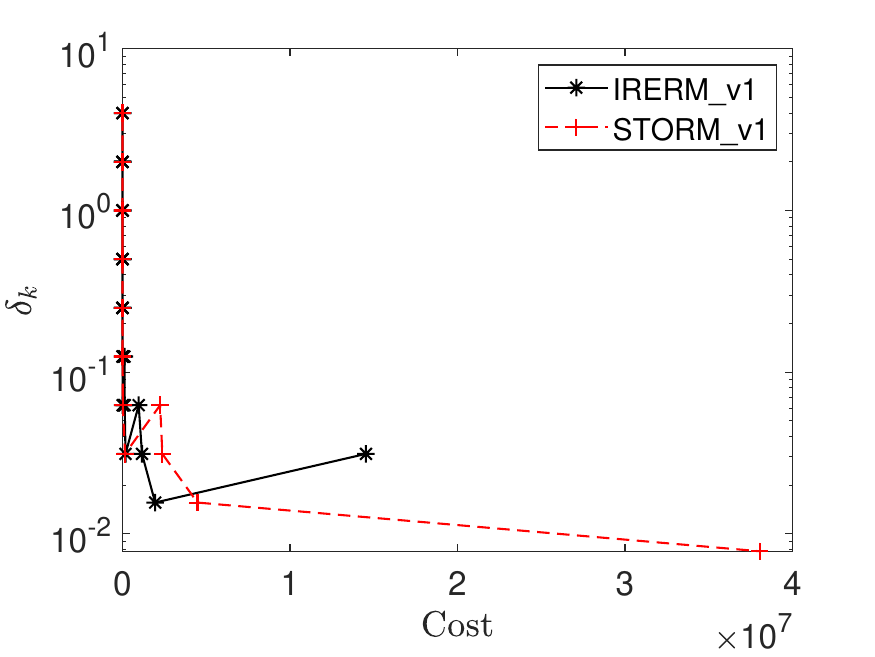}\\
\includegraphics[scale = 0.3]{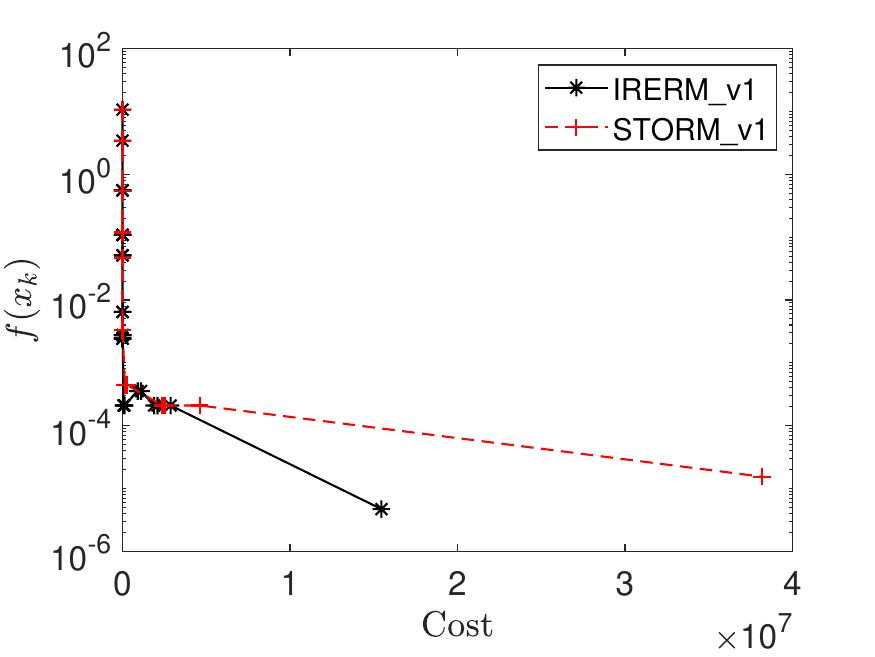}  & \includegraphics[scale = 0.3]{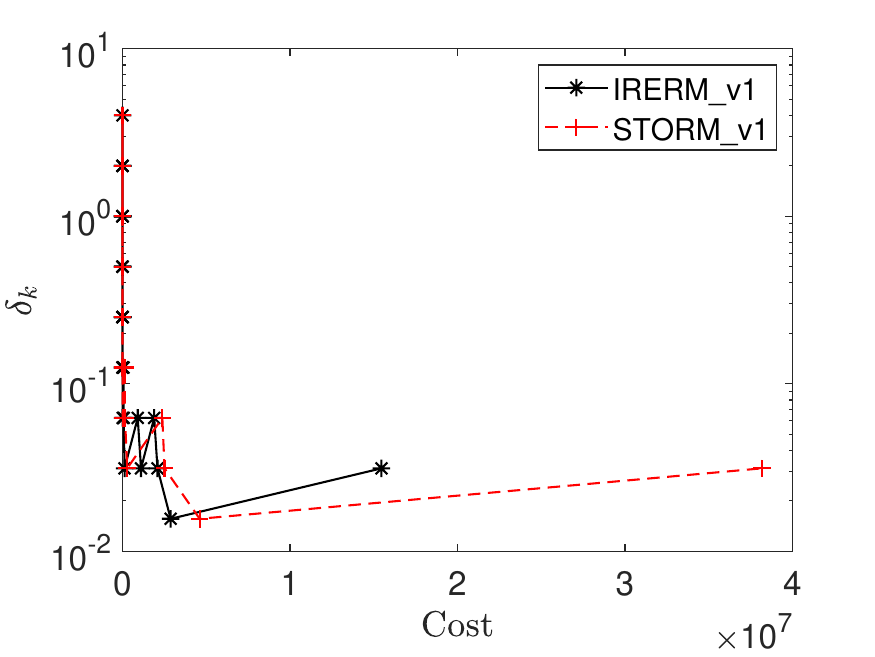}\\
\includegraphics[scale = 0.3]{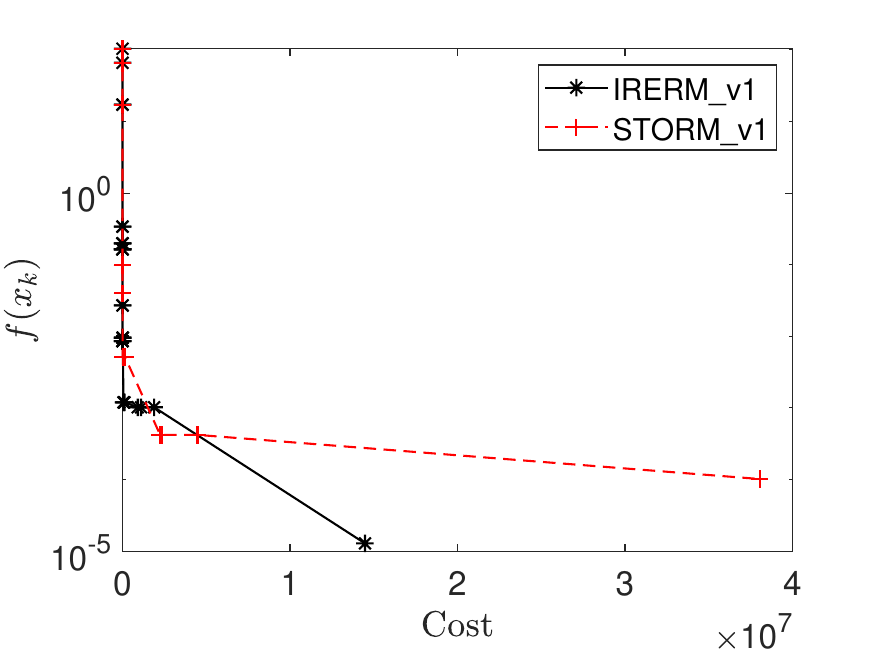}  & \includegraphics[scale = 0.3]{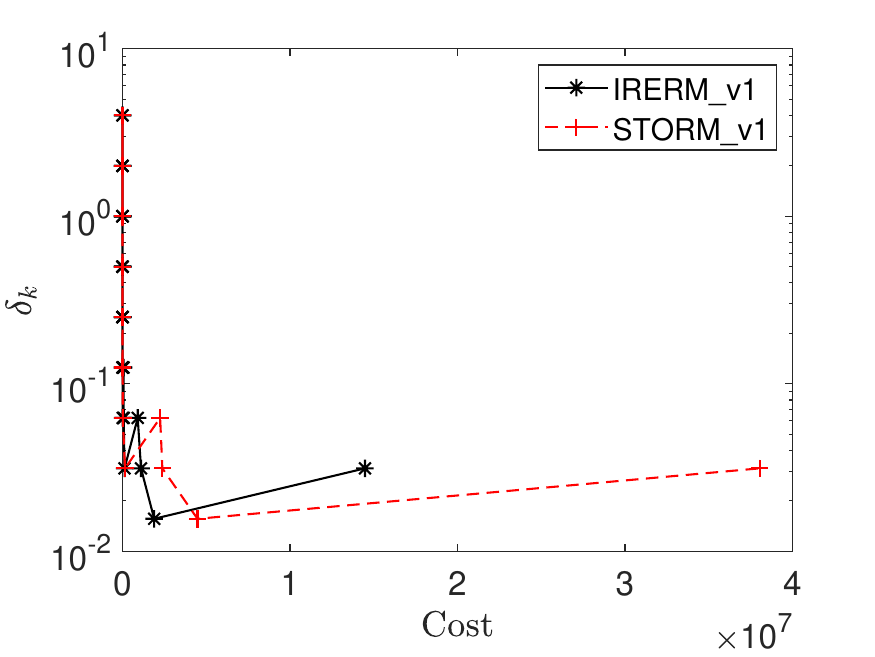}
\end{tabular}
\end{center}
\caption{Best runs of {\rm{\sc irerm}\_v1} and {\rm{\sc storm}\_v1} in terms of value of $f$ at termination. From left to right: value of true function $f$ vs computational budget and trust-region radius $\delta_k$ vs computational budget.
From top to bottom: Problems {\sc p5}, {\sc p10}, {\sc p11}, {\sc p12}.}\label{fig1}
\end{figure}

\begin{figure}[t!]
\begin{center}
\begin{tabular}{ccc}
\includegraphics[scale = 0.3]{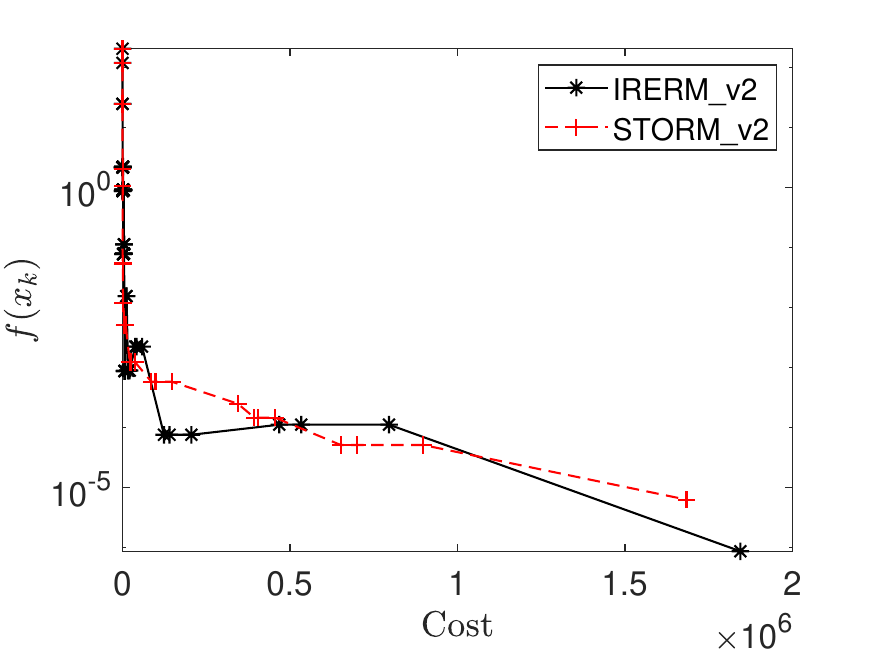}  & \includegraphics[scale = 0.3]{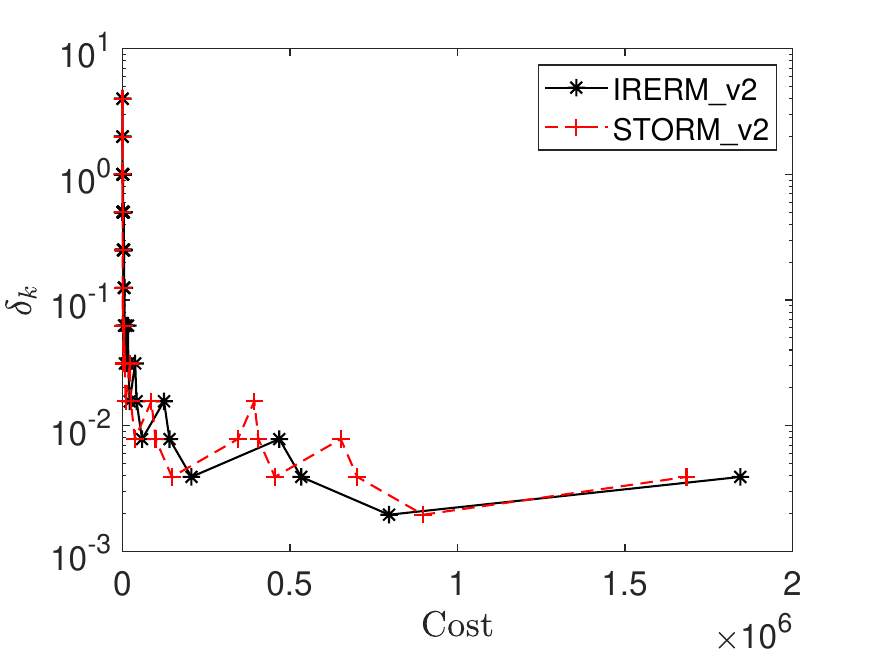}\\
\includegraphics[scale = 0.3]{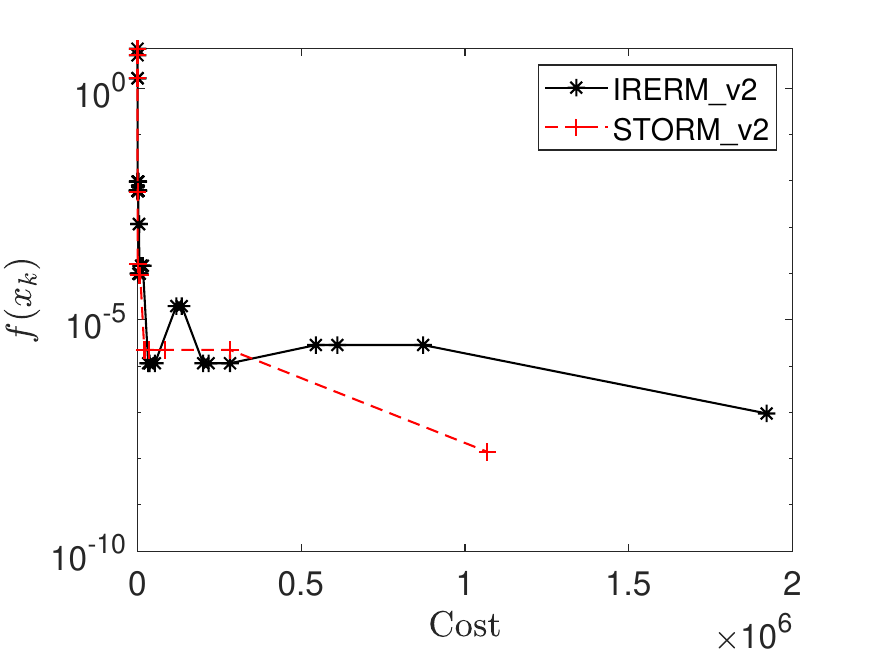}  & \includegraphics[scale = 0.3]{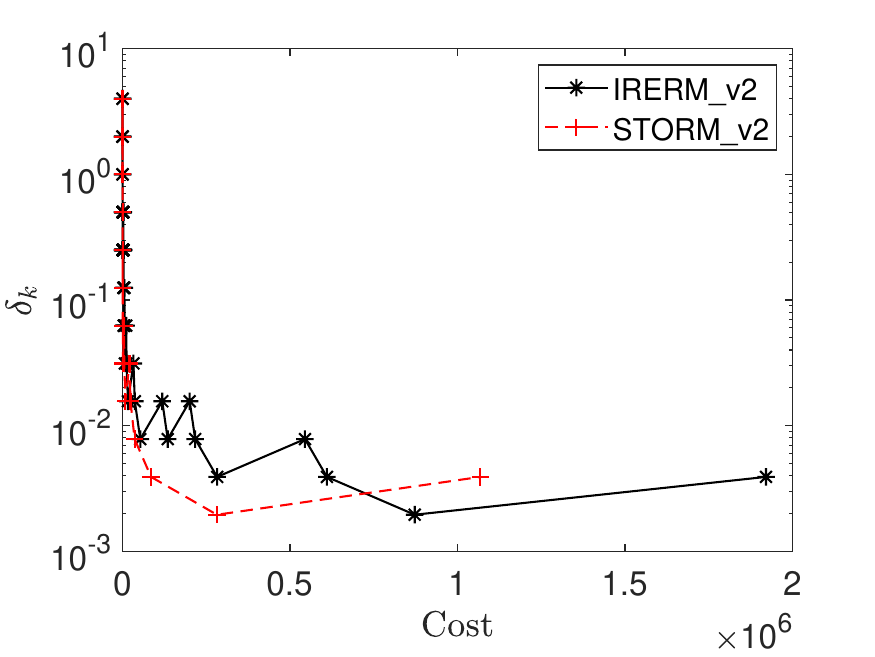}\\
\includegraphics[scale = 0.3]{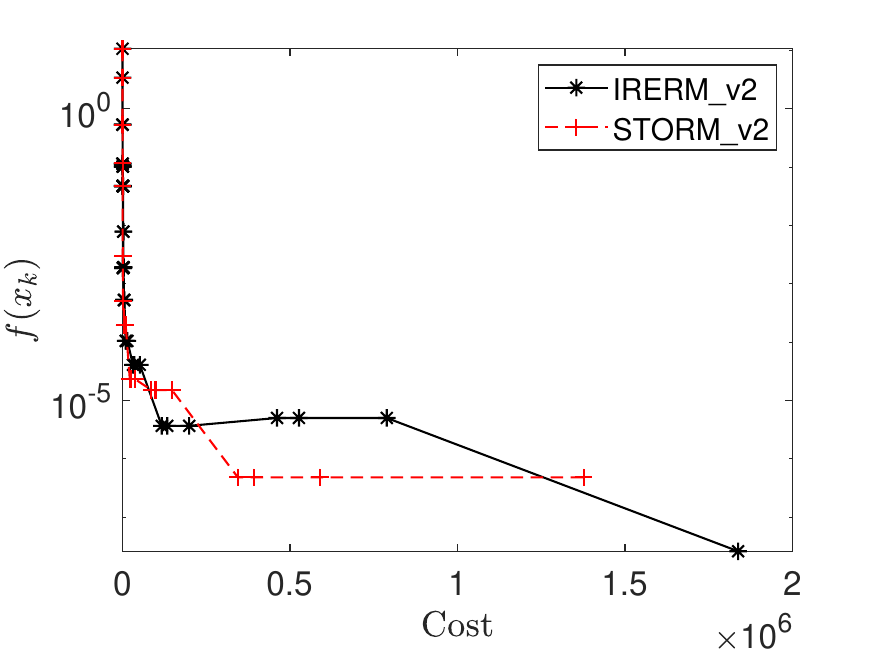}  & \includegraphics[scale = 0.3]{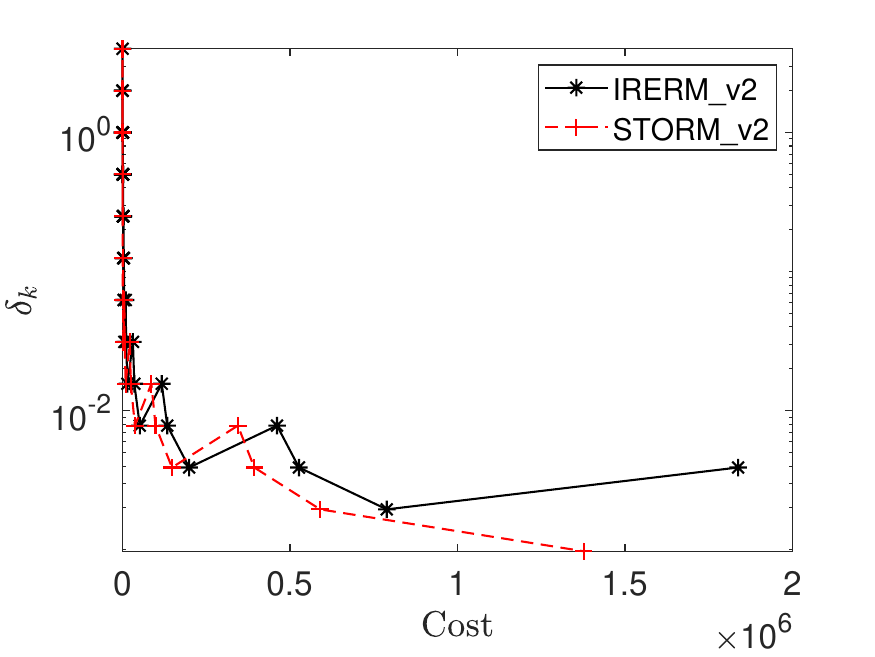}\\
\includegraphics[scale = 0.3]{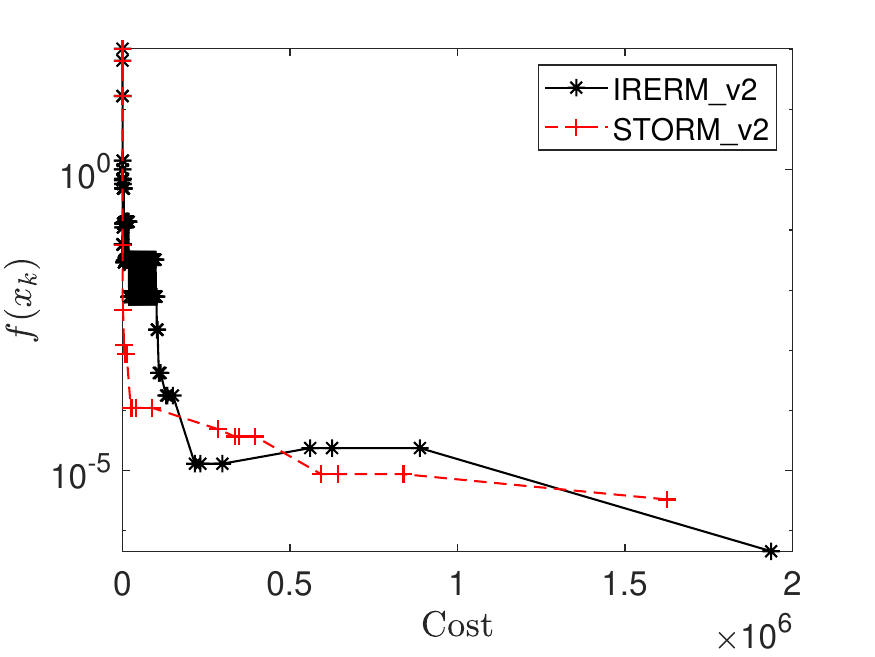}  & \includegraphics[scale = 0.3]{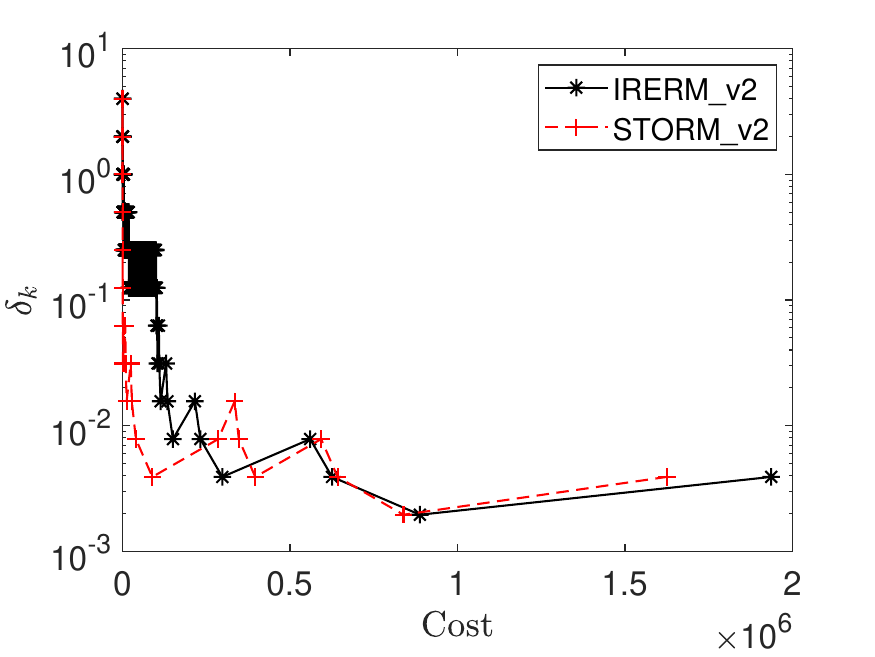}
\end{tabular}
\end{center}
\caption{Best runs of {\rm{\sc irerm}\_v2} and {\rm{\sc storm}\_v2} in terms of value of $f$ at termination. From left to right: value of true function $f$ vs computational budget and trust-region radius $\delta_k$ vs computational budget. From top to bottom: Problems {\sc p5}, {\sc p10}, {\sc p11}, {\sc p12}.}\label{fig2}
\end{figure}

Comparing the smallest value of $f$ achieved by {\sc irerm}\_v1  and {\sc storm}\_v1, we observe in Figure \ref{fig1}  
and Table \ref{tab2} that the two solvers show comparable performances over the selected problems. With respect to {\sc storm}\_v1,
{\sc irerm}\_v1 reaches the smallest value $9$ times out of $17$; notably, the minimum value obtained by {\sc irerm} is at least half of that of {\sc storm} on $4$ problems ({\sc p5}, {\sc p10}, {\sc p11}, {\sc p12}). On the other hand,   in the case where {\sc storm}\_v1 reaches the smallest value of $f$ at termination, such value is comparable with the value attained by {\sc irerm}\_v1, except for  problem {\sc p16}. In most iterations, the trust-region radius selected in {\sc irerm}\_v1 is larger than in {\sc storm}\_v1; this fact can be ascribed to
the acceptance rule of the iterates, which depends on the decrease of a convex combination of the values of $\bar f$ and $h$, instead of 
using only the values $\bar f$ as in {\sc storm}\_v1. We observe that larger trust-region radii generally imply looser accuracy requirements on functions and derivatives. This feature compensates the per iteration extra cost or {\sc irerm}  compared to the per iteration cost of {\sc storm} and makes {\sc irerm} competitive with {\sc storm}. The average value of $f$ reached at termination and the standard deviation of the two algorithms is summarized in Table \ref{tab1}. The average values obtained by {\sc irerm}\_v1 and {\sc storm}\_v1 are similar in most cases, with the exception of problems {\sc p6} and {\sc p8}.

Comparing the smallest value of $f$ achieved by {\sc irerm}\_v2  and {\sc storm}\_v2 displayed in Figure \ref{fig2}
and Table \ref{tab2}, we note that 
{\sc irerm}\_v2 reaches the smallest value $12$ times out of $17$ and compares well with {\sc storm}\_v2
in the remaining problems. In several iterations, the trust-region radius selected in {\sc irerm}\_v2 is larger than in {\sc storm}\_v2. Table \ref{tab1} shows that the average values  over 10 runs  are similar for {\sc irerm}\_v2 and {\sc storm}\_v2 except for problems {\sc p6} and {\sc p10}.   

The obtained results confirm that solving an unconstrained  noisy problem via a suitable constrained formulation and the Inexact Restoration approach is viable and compares well with {\sc storm} algorithm. Interestingly, the best results over multiple runs are good even for {\sc irerm\_v2}, which represents an heuristic version of our algorithm where the theoretical prescriptions are ignored. We also note that the proposed algorithm tends to compute larger trust-region radii than those of {\sc storm}, which leads in some cases to smaller values of $f$ at termination of the best runs.\\

\section{Conclusions and future work}
We proposed a trust-region algorithm with first-order random models suitable for the minimization of unconstrained noisy functions. The proposed algorithm, named {\sc irerm},  is based on a constrained reformulation of the minimization problem typical of the Inexact Restoration approach \cite{AEP, BM, BHM, FF}, where satisfying the constraint is equivalent to computing  both the function and its gradient exactly. To our knowledge, this is the first attempt to adopt  
the Inexact Restoration approach for a general problem where the evaluation of both the objective function and its gradient is random and sufficiently accurate in probability,
and to provide a theoretical analysis under conditions viable for the context. At each iteration, {\sc irerm} first enforces the sufficient accuracy in probability of the function and gradient estimates, and then employs an acceptance test involving both the noisy function and the infeasibility measure. Based on  
the theoretical analysis in \cite{bcms,ChenMeniSche18}, we derived an upper bound on the expected number of iterations needed to achieve an approximate first-order optimality point.  A preliminary numerical illustration  
of  {\sc irerm} on a collection of nonlinear least-squares problems
is also presented. As future work, we plan to deepen the numerical investigation of  {\sc irerm}  as well as design a generalized version of {\sc irerm} with second-order random models.

\section{Appendix}
This appendix is devoted to the proof of Theorem \ref{teo_limidelta0}, whose arguments follow closely those of \cite[Theorem 4.11]{ChenMeniSche18}. To this aim, we need the following four technical lemmas.

\begin{lemma}\label{lemtech1} 
Suppose Assumption \ref{ass_gradf} holds and that iteration $k$ is successful. Then we have 
\begin{equation}\label{eq:descent}
f(x_{k+1})-f(x_k)\leq \|\nabla f(x_k)\|\delta_k+L\delta_k^2.
\end{equation}
\end{lemma}
\begin{proof}
If the iteration is successful, then $x_{k+1}=x_k+p_k$. Applying the Descent Lemma for functions with Lipschitz continuous gradient to $f$ and using  $p_k=-\delta_k g_k/\|g_k\|$ yields 
$$
f(x_{k+1})-f(x_k)\leq \nabla f(x_k)^Tp_k+L\|p_k\|^2\leq \|\nabla f(x_k)\|\delta_k+L\delta_k^2. 
$$
\end{proof}
\vskip 5pt
\begin{lemma}\label{lemtech2}
Suppose Assumption \ref{ass_fh} holds and let  $\psi_k$ be defined as in (\ref{eq:Lyapunov}) and (\ref{eq:phik}), with the constant $\Sigma$ chosen as in (\ref{eq:sigma}). 
The following facts hold true.
\\
$(i)$ It holds
\begin{align}\label{eq:phi_diff}
\psi_{k+1}-\psi_k&\leq v\left(\theta_{k+1} ({ f(x_{k+1})}-{ f(x_{k})})+(1-\theta_{k+1})(h(y_{k+1})-h(y_k))\right)\nonumber\\
&\quad +(1-v)(\delta_{k+1}^2-\delta_k^2).
\end{align}
\vskip 5pt \noindent
$(ii)$ If iteration $k$ is true and successful, then
\begin{equation}\label{eq:phi_diff_3}
\psi_{k+1}-\psi_k\leq -v\eta_1 \theta_{k+1}\delta_k\|g_k\|+(1-v)(\gamma^2-1)\delta_k^2+2v \theta_{k+1}\kappa\delta_k^2.
\end{equation}
\vskip 5pt \noindent
$(iii)$ If iteration $k$ is unsuccessful, then
\begin{equation}\label{eq:unsucc_2}
\psi_{k+1}-\psi_k\leq (1-v)\left(\frac{1-\gamma^2}{\gamma^2}\right)\delta_k^2<0.
\end{equation}
\end{lemma}
\begin{proof}{} (i) The successive difference $\psi_{k+1}-\psi_k$ can be written as 
\begin{align}
\psi_{k+1}-\psi_k&= v(\theta_{k+1}  f(x_{k+1}) +(1-\theta_{k+1})h(y_{k+1}))\nonumber\\
&\quad -v(\theta_{k}{ f(x_k)}+(1-\theta_{k})h(y_{k}))+v(\theta_{k+1}-\theta_k)\Sigma+(1-v)(\delta_{k+1}^2-\delta_k^2)\nonumber\\
&= v\left(\theta_{k+1} ({ f(x_{k+1})}-{ f(x_{k})})+(1-\theta_{k+1})(h(y_{k+1})-h(y_k))\right)\nonumber\\
&\quad +v(\theta_{k+1}-\theta_k)({ f(x_k)}-h(y_k)+\Sigma)+(1-v)(\delta_{k+1}^2-\delta_k^2)\nonumber\\
&\leq v\left(\theta_{k+1} ({ f(x_{k+1})}-{ f(x_{k})})+(1-\theta_{k+1})(h(y_{k+1})-h(y_k))\right)\nonumber\\
&\quad +(1-v)(\delta_{k+1}^2-\delta_k^2),\nonumber
\end{align}
where the last inequality follows from \eqref{eq:sigma} and the monotonicity 
of $\{\theta_k\}$ (see Lemma \ref{lem:thetak}(i)). 

(ii) The occurrence of successful iteration \eqref{eq:succ} and inequality \eqref{eq:phi_diff} yield
\begin{align*}
\psi_{k+1}-\psi_k&\leq v( \theta_{k+1}(f(x_k+p_k)-f(x_k))+(1-\theta_{k+1})(h(y_{k+1})-h(y_k)))\\
&+(1-v)(\gamma^2-1)\delta_k^2.
\end{align*}
By summing and subtracting the quantities $v \theta_{k+1}\bar{f}_k^*$ and $v \theta_{k+1}\bar f_k^p$ in the above inequality, we get
\begin{align*}
\psi_{k+1}-\psi_k&\leq v(\theta_{k+1}( \bar f_k^p-\bar{f}_k^*)+(1- \theta_{k+1})(h( y_{k+1})-h(y_k)))+(1-v)(\gamma^2-1)\delta_k^2\\
& \ \ +v \theta_{k+1} (f(x_k+p_k)- \bar f_k^p+\bar f_k^*-f(x_k))\\
&\leq v( \theta_{k+1} ( \bar f_k^p-\bar f_k^*)+(1- \theta_{k+1})(h( y_{k+1})-h(y_k)))+(1-v)(\gamma^2-1)\delta_k^2\\
&\ \ +v \theta_{k+1}(|f(x_k+p_k)- \bar{f}_k^p|+|\bar{f}_k^*-f(x_k)|).
\end{align*}
By the definition of $\Aredk$ in \eqref{eq:ared},
the fact that $y_{k+1}=y_{k+1}^t$ due to $k$ being successful, and  \eqref{eq:true1}-\eqref{eq:true2}, we get
\begin{equation}\label{eq:phi_diff_2}
\psi_{k+1}-\psi_k\leq -v\Aredk(x_k+p_k,\theta_{k+1})+(1-v)(\gamma^2-1)\delta_k^2+2v \theta_{k+1}\kappa\delta_k^2.
\end{equation}
Since we know that \eqref{eq:ared_successfulbis} holds for successful iterations, from the previous inequality we easily get  
\eqref{eq:phi_diff_3}.

(iii)
If $k$ is unsuccessful, the  claim follows from \eqref{eq:phi_diff} and  $x_{k+1}=x_k$, $y_{k+1}=y_{k}$, $\delta_{k+1}= \delta_k/\gamma$.
\end{proof}

\vskip 5pt

To proceed, we distinguish the case $\|\nabla f(X_k)\|\geq \zeta \Delta_k$ from the case $\|\nabla f(X_k)\|< \zeta \Delta_k$, where $\zeta$ is defined in (\ref{eq:zeta_condition}).
\vskip 5pt
\noindent
\begin{lemma}\label{lemtech3}
Under the Assumptions of Theorem \ref{teo_limidelta0}, it holds
\begin{equation}\label{eq:phi_diff_8}
\Ek[\Psi_{k+1}-\Psi_k\, |\, \|\nabla f(X_k)\|\geq \zeta \Delta_k] \leq -(1-v)(\gamma^2-1)\Delta_k^2,
\end{equation}
with $\zeta$ as in (\ref{eq:zeta_condition}).
\end{lemma}
\begin{proof}
Let $x_k$, $\delta_k$ and $\psi_k$ denote the realizations of the random variables $X_k, \Delta_k$ and  $\Psi_k$, respectively.
Consider an arbitrary realization of Algorithm \nome and the case $\|\nabla f(x_k)\|\geq \zeta \delta_k$. We analyze separately true and false iterations.
If the iteration $k$ is true, we can combine \eqref{eq:true3} with $\|\nabla f(x_k)\|\geq \zeta \delta_k$ to obtain
\begin{equation}\label{eq:gk_succ}
\|g_k\|\geq \|\nabla f(x_k)\|-\kappa\delta_k\geq (\zeta-\kappa)\delta_k\geq \max\left\{\eta_2,\frac{\theta_0(3\kappa+L)+\mu}{\underline{\theta}(1-\eta_1)}\right\}\delta_k,
\end{equation} 
where the last inequality is due to condition \eqref{eq:zeta_condition} and the upper bound $\mu\leq \kappa \bar{\mu}_{up}$, which follows from the hypothesis $\mu\leq \kappa\bar{\mu}(\pi^{\frac{1}{4}})$ and Assumption \ref{ass_fh}(ii). Since $k$ is true, $\kappa$ satisfies condition \eqref{eq:kappa_condition}, and inequality \eqref{eq:gk_succ} holds, Lemma \ref{lemma_succ} implies that $k$ is successful and consequently \eqref{eq:phi_diff_3}  holds.

By exploiting once again the fact that $k$ is true and $\|\nabla f(x_k)\|\geq \zeta \delta_k$, we observe that
\begin{equation*}
\|g_k\|\geq \|\nabla f(x_k)\|-\kappa \delta_k\geq \left(1-\frac{\kappa}{\zeta}\right)\|\nabla f(x_k)\|,
\end{equation*}
and \eqref{eq:phi_diff_3} yields 
\begin{align}
\psi_{k+1}-\psi_k
&\leq -v\eta_1 \theta_{k+1} \left(1-\frac{\kappa}{\zeta}\right)\|\nabla f(x_k)\|\delta_k+(1-v)(\gamma^2-1)\delta_k^2+\frac{2v \theta_{k+1}\kappa\delta_k\|\nabla f(x_k)\|}{\zeta} \nonumber \\
& =v \theta_{k+1}\left(\frac{2\kappa}{\zeta}-\eta_1\left(1-\frac{\kappa}{\zeta}\right)\right)\|\nabla f(x_k)\|\delta_k+(1-v)(\gamma^2-1)\delta_k^2. \nonumber
\end{align}
Note that $\frac{2\kappa}{\zeta}-\eta_1\left(1-\frac{\kappa}{\zeta}\right) < 0$
due to  \eqref{eq:zeta_condition}, and $\theta_{k+1}=\theta_{k+1}^t\geq \underline{\theta}$ since the acceptance condition \eqref{eq:accept3} holds when $k$ is successful. Then, by the definition of $C_1$ in \eqref{eq:C1}, it follows
\begin{equation}\label{eq:phi_diff_5}
\psi_{k+1}-\psi_k\leq -vC_1\delta_k\|\nabla f(x_k)\|+(1-v)(\gamma^2-1)\delta_k^2,
\end{equation}
and by $\|\nabla f(x_k)\|\geq \zeta \delta_k$ and (\ref{eq:v_condition}) we have that 
$$
-vC_1\delta_k\|\nabla f(x_k)\|+(1-v)(\gamma^2-1)\delta_k^2\le (-vC_1\zeta+(1-v)(\gamma^2-1))\delta_k^2<0.
$$

If the iteration $k$ is false, $k$ can be either a successful or an unsuccessful iteration. 
If $k$ is successful, using \eqref{eq:descent} and $\|\nabla f(x_k)\|\geq \zeta \delta_k$ we obtain
\begin{equation}\label{eq:bound_lemma}
f(x_{k+1})-f(x_k)\leq \left(1+\frac{L}{\zeta}\right)\|\nabla f(x_k)\|\delta_k.
\end{equation}
Furthermore, we have
\begin{equation}\label{eq:bound_h}
h(y_{k+1})-h(y_k)=h(y_{k+1}^t)-h(y_k)\leq \mu \delta_k^2,
\end{equation}
due to \eqref{updateh}. By plugging \eqref{eq:bound_lemma} and \eqref{eq:bound_h} inside \eqref{eq:phi_diff}, we get
\begin{align}\label{eq:phi_diff_6}
\psi_{k+1}-\psi_k&\leq v\left(\theta_{k+1}\left(1+\frac{L}{\zeta}\right)\|\nabla f(x_k)\|\delta_k+(1-\theta_{k+1})\mu\delta_k^2\right)+(1-v)(\gamma^2-1)\delta_k^2
\nonumber\\
&\leq v C_2\|\nabla f(x_k)\|\delta_k+(1-v)(\gamma^2-1)\delta_k^2,
\end{align}
where the second inequality follows from $\|\nabla f(x_k)\|\geq \zeta\delta_k$, the bounds $0<\theta_{k+1}< 1$, the hypothesis $\mu\leq \kappa \bar{\mu}(\pi^{\frac{1}{4}})\leq \kappa \bar{\mu}_{up}$ and the definition of $C_2$ in \eqref{eq:C1}.

If $k$ is unsuccessful, then \eqref{eq:unsucc_2} holds. Since the right-hand side of \eqref{eq:phi_diff_6} is non-negative, we conclude that inequality \eqref{eq:phi_diff_6} holds either if $k$ is successful or unsuccessful.

Let $\Psi_k$ be the random variable with realization $\psi_k$.
We are now ready to consider the expected value of $\Psi_{k+1}-\Psi_k$ conditioned to the past and the event $\{\|\nabla f(X_k)\|\geq \zeta \Delta_k\}$. By recalling  Lemma \ref{lem_prob} and combining \eqref{eq:phi_diff_5} and \eqref{eq:phi_diff_6}, we obtain
\begin{eqnarray}
& & \Ek[\Psi_{k+1}-\Psi_k\, |\, \|\nabla f(X_k)\| \geq \zeta \Delta_k]   \nonumber \\
& & \quad \qquad \leq \pi(-vC_1\Delta_k\|\nabla f(X_k)\|+(1-v)(\gamma^2-1)\Delta_k^2)+\nonumber\\
&& \quad \qquad  (1-\pi)(vC_2\|\nabla f(X_k)\|\Delta_k+(1-v)(\gamma^2-1)\Delta_k^2)\nonumber\\
& & \quad \qquad  = v\|\nabla f(X_k)\|\Delta_k(-C_1\pi+(1-\pi)C_2)+(1-v)(\gamma^2-1)\Delta_k^2.\label{eq:phi_diff_7}
\end{eqnarray}
The choice of  $v$ in \eqref{eq:v_condition} and the choice of  $\pi$ in \eqref{eq:pi_condition} imply
$\displaystyle\frac{C_1}{2}\geq \frac{2\gamma^2(1-v)}{v\zeta}$ and 
$\displaystyle C_1\pi - (1-\pi)C_2\geq \frac{C_1}{2}$, respectively. Consequently
\begin{equation}\label{eq:tech_1}
C_1\pi - (1-\pi)C_2\geq \frac{2(\gamma^2-1)(1-v)}{v\zeta}.
\end{equation}
Thus \eqref{eq:phi_diff_7} leads to
\begin{align*}
\Ek[\Psi_{k+1}-\Psi_k\, |\, \|\nabla f(X_k)\|\geq \zeta \Delta_k]
&\leq-\frac{2(1-v)(\gamma^2-1)}{\zeta}\|\nabla f(X_k)\|\Delta_k \\
& \qquad +(1-v)(\gamma^2-1)\Delta_k^2\\
&\leq -2(1-v)(\gamma^2-1)\Delta_k^2+(1-v)(\gamma^2-1)\Delta_k^2\\
&=-(1-v)(\gamma^2-1)\Delta_k^2,
\end{align*}
where the third inequality is again due to $\|\nabla f(X_k)\|\geq \zeta \Delta_k$.
\end{proof}

\vskip 5pt
Now we analyze the case $\|\nabla f(X_k)\|< \zeta \Delta_k$, where $\zeta$ is again defined in (\ref{eq:zeta_condition}).
\vskip 5pt
\noindent
\begin{lemma}\label{lemtech4}
Under the Assumptions of Theorem \ref{teo_limidelta0}, it holds
\begin{equation}\label{eq:phi_diff_12}
\Ek[\Psi_{k+1}-\Psi_k\, |\, \|\nabla f(X_k)\|<\zeta\Delta_k]\leq -\frac{1}{2}(1-v)\left(\frac{\gamma^2-1}{\gamma^2}\right)\Delta_k^2,
\end{equation}
with $\zeta$ as in (\ref{eq:zeta_condition}).
\end{lemma}
\begin{proof}
Let $x_k$, $\delta_k$ and $\psi_k$ denote the realizations of the random variables $X_k, \Delta_k$ and  $\Psi_k$, respectively.
Consider an arbitrary realization of Algorithm \nome and the case $\|\nabla f(x_k)\|< \zeta \delta_k$. We analyze separately true and false iterations.

If the iteration $k$ is true and  successful, then inequality (\ref{eq:phi_diff_3}) holds.
Thus, (\ref{eq:phi_diff_3}) and condition (\ref{eq:accept2})  
yield
\begin{align*}
\psi_{k+1}-\psi_k&\leq -v\eta_1\eta_2 \theta_{k+1}\delta_k^2+(1-v)(\gamma^2-1)\delta_k^2+2v \theta_{k+1}\kappa\delta_k^2\\
& = v \theta_{k+1}(2\kappa-\eta_1\eta_2)\delta_k^2+(1-v)(\gamma^2-1)\delta_k^2.
\end{align*} 
We observe that $2\kappa-\eta_1\eta_2 < 0$ by \eqref{eq:kappa_condition} and $\theta_{k+1}=\theta_{k+1}^t\geq \underline{\theta}$ due to the acceptance condition \eqref{eq:accept3} holding for successful iterations.  Then, by recalling the definition of $C_3$ in \eqref{eq:C1}, we get
\begin{equation}\label{eq:phi_diff_9}
\psi_{k+1}-\psi_k\leq 
(-vC_3+(1-v)(\gamma^2-1))\delta_k^2.
\end{equation}
The form of $v$ in \eqref{eq:v_condition} gives 
$vC_3\geq 2\gamma^2(1-v)$ and consequently $vC_3\geq 2(\gamma^2-1)(1-v)$, then  \eqref{eq:phi_diff_9} leads to
\begin{equation*}
\psi_{k+1}-\psi_k\leq -(1-v)(\gamma^2-1)\delta_k^2.
\end{equation*}
If $k$ is unsuccessful, then the sufficient decrease of $\psi_{k+1}-\psi_k$ is guaranteed by \eqref{eq:unsucc_2}. Since $(\gamma^2-1)/\gamma^2<\gamma^2-1$, we conclude that \eqref{eq:unsucc_2} holds for  all true iterations.

If  iteration $k$ is false and  successful,  (\ref{eq:descent})  and $\|\nabla f(x_k)\|< \zeta \delta_k$ give
\begin{align}\label{eq:bound_lemma_2}
f(x_{k+1})-f(x_k)&
\leq \left(\zeta+L\right)\delta_k^2.
\end{align} 
Therefore,  \eqref{eq:phi_diff}, \eqref{eq:bound_h} and \eqref{eq:bound_lemma_2} yield to
\begin{align}\label{eq:phi_diff_10}
\psi_{k+1}-\psi_k
& \leq v(\theta_{k+1}(L+\zeta)\delta_k^2+(1-\theta_{k+1})\mu\delta_k^2)+(1-v)(\gamma^2-1)\delta_k^2\nonumber\\
&\leq (vC_4+(1-v)(\gamma^2-1))\delta_k^2,
\end{align}
where the second inequality follows from the fact that $0<\theta_{k+1}\leq 1$, the hypothesis $\mu\leq \kappa \bar{\mu}(\pi^{\frac{1}{4}})\leq \kappa \bar{\mu}_{up}$ and the definition of $C_4$ given in \eqref{eq:C1}.

If $k$ is unsuccessful, then \eqref{eq:unsucc_2} holds, and thus $\psi_{k+1}-\psi_k<0$. 
Since the right-hand side of \eqref{eq:phi_diff_10} is non-negative, we conclude that inequality \eqref{eq:phi_diff_10} holds    for all false iterations.

By taking the expected value of $\Psi_{k+1}-\Psi_k$ conditioned to the past and the event $\|\nabla f(X_k)\|< \zeta\Delta_k$, and combining \eqref{eq:unsucc_2} with \eqref{eq:phi_diff_10}, we obtain
\begin{align*}\label{eq:phi_diff_11}
& \Ek[\Psi_{k+1}-\Psi_k\, |\, \|\nabla f(X_k)\|<\zeta\Delta_k]\\
& \qquad \leq \pi(1-v)\left(\frac{1-\gamma^2}{\gamma^2}\right)\Delta_k^2+(1-\pi)(vC_4+(1-v)(\gamma^2-1))\Delta_k^2\nonumber\\
&\qquad =(\pi\pm 1)(1-v)\left(\frac{1-\gamma^2}{\gamma^2}\right)\Delta_k^2+(1-\pi)(vC_4+(1-v)(\gamma^2-1))\Delta_k^2\nonumber\\
&\qquad =(1-v)\left(\frac{1-\gamma^2}{\gamma^2}\right)\Delta_k^2+(1-\pi)\left(vC_4+(1-v)\left(\frac{\gamma^4-1}{\gamma^2}\right)\right)\Delta_k^2.
\end{align*}
Using (\ref{eq:v_condition}), the choice of $\pi$ in \eqref{eq:pi_condition} gives
$(1-\pi)\left(vC_4+(1-v)\left(\frac{\gamma^4-1}{\gamma^2}\right)\right)\leq \frac{1}{2}(1-v)\left(\frac{\gamma^2-1}{\gamma^2}\right)$ and the claim follows.
\end{proof}

We conclude by deriving the proof of Theorem \ref{teo_limidelta0}.
\vskip 5pt
\begin{proof}[Proof of Theorem \ref{teo_limidelta0}]
By combining \eqref{eq:phi_diff_8} and \eqref{eq:phi_diff_12} and observing that $(\gamma^2-1)/(2\gamma^2)<\gamma^2-1$, we  obtain
$$
\Ek[\Psi_{k+1}-\Psi_k]\leq -\frac 1 2 (1-v)\left(\frac{\gamma^2-1}{\gamma^2}\right)\Delta_k^2,
$$
which means that inequality \eqref{eq:suff_decrease} holds with   $\sigma = \frac 1 2 (1-v)(\gamma^2-1)/\gamma^2$.

In order to prove \eqref{eq:delta_summable}, we sum \eqref{eq:suff_decrease} over $k=0,\ldots,K$ to obtain
\begin{equation*}
\sum_{k=0}^K\sigma\Delta_k^2\leq \sum_{k=0}^K\Ek[\Psi_k-\Psi_{k+1}].
\end{equation*}
By taking the total expected value, using the linearity of expectation and the law of total expectation, we obtain 
\begin{align*}
\mathbb{E} \left[\sum_{k=0}^K\sigma\Delta_k^2\right]
&\leq\sum_{k=0}^{K}\mathbb{E}\left[\Ek[\Psi_k-\Psi_{k+1}]\right]\\
&= \mathbb{E}\left[\sum_{k=0}^{K}(\Psi_k-\Psi_{k+1})\right]\\
&\leq \mathbb{E}[\Psi_0-\Psi_{low}], 
\end{align*}
where the last inequality is due to the fact that $\{\Psi_k\}$ is bounded from below by a constant $\Psi_{low}$ for all $k\geq 0$, thanks to Assumption \ref{ass_fh}(i) and (iii). By taking the limit for $K\rightarrow \infty$ and applying the monotone convergence theorem on the left-hand side of the inequality, we get
$\mathbb{E}\left[\sum_{k=0}^{\infty}\Delta_k^2\right]<\infty$, from which \eqref{eq:delta_summable} immediately follows.
\end{proof}

\bibliographystyle{amsplain}

\end{document}